\documentclass[10pt]{amsart}

\usepackage{lmodern}\usepackage[T1]{fontenc}

\usepackage{amsmath,amssymb,amsthm,graphicx,epstopdf,mathrsfs,url}
\usepackage[T1]{fontenc}
\usepackage[usenames,dvipsnames]{color}
\usepackage{dsfont}   
\definecolor{darkred}{rgb}{0.4,0.1,0.1}
\definecolor{darkblue}{rgb}{0.1,0.1,0.4}
\definecolor{darkgrey}{rgb}{0.5,0.5,0.5}
\usepackage[colorlinks=true,linkcolor=darkred,citecolor=Blue]{hyperref}

\def\arr{\rightarrow}

\def\tt{\theta}
\def\aa{\alpha}
\def\lm{\lambda}

\def\S{\Sigma}

\def\s{\sigma}
\def\sd{\sigma_{\rm d}}
\def\sess{\sigma_{\rm ess}}

\def\sfh{\mathsf{h}}

\def\dd{{\mathrm{d}}}

\newcounter{counter_a}
\newenvironment{myenum}{\begin{list}{{\rm(\roman{counter_a})}}%
{\usecounter{counter_a}
\setlength{\itemsep}{1.ex}\setlength{\topsep}{0.8ex}
\setlength{\leftmargin}{5ex}\setlength{\labelwidth}{5ex}}}{\end{list}}

\newcommand{\cf}{{cf.}\,}       

\numberwithin{figure}{section}
\numberwithin{equation}{section}
\theoremstyle{plain}
\newtheorem*{thm*}{Theorem}
\newtheorem{thm}{Theorem}[section]
\newtheorem{hyp}{Hypothesis}[section]
\newtheorem{lem}[thm]{Lemma}
\newtheorem{prop}[thm]{Proposition}
\newtheorem{example}[thm]{Example}

\theoremstyle{remark}

\theoremstyle{plain}
\newtheorem{hypothesis}[thm]{Hypothesis}

\newcommand{\rmd}{\mathrm{d}}

\newcommand{\supp}{\mathrm{supp}\,}
\newcommand{\beu}{\begin{equation*}}
\newcommand{\eeu}{\end{equation*}}
\newcommand{\besu}{\begin{equation*}
\begin{aligned}}
\newcommand{\eesu}{\end{aligned}
\end{equation*}}
\newcommand{\bes}{\begin{equation}
\begin{aligned}}
\newcommand{\ees}{\end{aligned}
\end{equation}}

\newcommand\cL{\mathcal L}

\newcommand\fra{\mathfrak a}

\newcommand\eps{\varepsilon}

\newcommand\ov{\overline}
\newcommand\wt{\widetilde}
\newcommand\wh{\widehat}

\newcommand\sign{{\rm sign\,}}

\newcommand\void[1]{}

\def\ov{\overline}
\def\eps{\varepsilon}
\def\ran{{\rm ran\,}}

      \def\dC{{\mathbb C}}

   \def\dN{{\mathbb N}}   
      \def\dR{{\mathbb R}}

      \def\cC{{\mathcal C}}
      
      \def\cI{{\mathcal I}}
      \def\cL{{\mathcal L}}

\newcommand{\dom}{\mathrm{dom}\,}

\renewcommand{\hyp}[1]{$C^{2}$-hypersurface as in Definition~\ref{definition_hypersurface}}

\newtheorem{definition}[thm]{Definition}

\newcommand{\skipex}{\setlength{\parskip}{0.8ex}}

\newcommand{\rank}{\mathrm{rank}\,}

\newcommand{\spann}{\mathrm{span}\,}
\begin{document}
\title[]{Approximation of Schr\"{o}dinger operators with \boldmath{$\delta$}-interactions supported on hypersurfaces}

\author{Jussi Behrndt}
\address{
Institut f\"{u}r Numerische Mathematik\\
Graz University of Technology\\
Steyrergasse 30\\
8010 Graz\\ Austria}
\email{behrndt@tugraz.at}

\author{Pavel Exner}
\address{Doppler Institute for Mathematical Physics and Applied Mathematics\\ 
Czech Technical University in Prague,
B\v{r}ehov\'{a} 7, 11519 Prague \\ and  Nuclear Physics Institute, Czech Academy of Sciences,
25068 \v{R}e\v{z} near Prague, Czechia}
\email{exner@ujf.cas.cz} 

\author{Markus Holzmann}
\address{Institut f\"{u}r Numerische Mathematik\\
Graz University of Technology\\
Steyrergasse 30\\
8010 Graz\\ Austria}
\email{holzmann@math.tugraz.at}

\author{Vladimir Lotoreichik}
\address{Department of Theoretical Physics,
Nuclear Physics Institute, Czech Academy of Sciences, 250 68, 
\v{R}e\v{z} near Prague, Czechia}
\email{lotoreichik@ujf.cas.cz}

\begin{abstract}
We show that a Schr\"odinger operator $A_{\delta, \alpha}$ with a $\delta$-interaction of strength $\alpha$ supported on a bounded or unbounded
$C^2$-hypersurface $\Sigma \subset \mathbb{R}^d$, $d\ge 2$, 
can be approximated in the norm resolvent sense 
by a family of Hamiltonians with suitably scaled regular potentials. The  differential operator $A_{\delta, \alpha}$ with a singular interaction
is regarded as a self-adjoint realization of the formal differential expression $-\Delta - \alpha \langle \delta_{\Sigma}, \cdot \rangle \delta_{\Sigma}$, 
where $\alpha\colon\Sigma\rightarrow \dR$ is an arbitrary bounded measurable function. We discuss also some spectral consequences of this approximation result.
\end{abstract}



\maketitle

\section{Introduction}

Singular Schr\"odinger operators with `potentials' supported by subsets of the configuration space of a lower dimension are often used as models 
of physical systems because they are easier to solve, the original differential equation being reduced to the analysis of an algebraic or functional 
problem. The best known about them are solvable models with point interactions used in physics since the 1930s (see \cite{de_kronig_penney_1931}), 
the rigorous analysis of which started from the seminal paper \cite{BF61}; 
for a survey see the monograph \cite{AGHH}. In the last two decades the attention focused on interactions 
supported on curves, surfaces, and more complicated sets composed of them, which are used to model `leaky' quantum systems in which the particle 
is confined to such manifolds or complexes, but the tunnelling between different parts of the interaction support is not neglected; for a review see \cite{E08} or 
\cite[Chapter 10]{EK15}.

While these models are useful and mathematically accessible, one has to keep in mind that the singular interaction represents an idealized 
form of the actual, more realistic description. This naturally inspires the question about approximations of such singular potentials by 
regular ones. In the simplest case of a point interaction this problem was already addressed in the 1930s in
\cite{thomas_1935}. Starting from the 1970s the approximation of Hamiltonians with point interactions supported on a finite or an infinite 
set of points in~$\mathbb{R}^d$, $d \in \{ 1, 2, 3 \}$, was treated systematically; cf. the monograph \cite{AGHH} and the references therein.

Apart from that, the literature on the approximation of Schr\"odinger operators with $\delta$-potentials supported on curves in $\mathbb{R}^2$
and surfaces in $\mathbb{R}^3$ is less complete; there are results available for the special cases that~$\Sigma$ is a sphere in 
$\dR^3$ \cite{antoine_gesztesy_shabani_1987, S92}, that~$\Sigma$
is the boundary of a star-shaped domain in the plane~\cite{P95},
and that~$\Sigma$ is a smooth planar curve or surface and the interaction strength is constant \cite{EI01, EK03}.
In all of the above mentioned works convergence in the norm resolvent sense is shown. Abstract approaches developed in~\cite{AK99, stollmann_voigt1996}
cover more cases but imply only strong resolvent convergence in this context.
We point out that the usage of scaled regular potentials
is not the unique way of an approximation of $\delta$-interactions
supported on hypersurfaces, other mechanisms of approximation
are discussed in e.g. \cite{BFT98, BO07, EN03, O06}.
It is also worth mentioning that the approximation of $\delta$-interactions
supported on special periodic structures in~$\dR^2$ 
has important applications in the mathematical theory of 
photonic crystals, see~\cite{FK98} and the references therein. 

The aim of the present paper is to analyze the general case 
where the interaction support is a $C^2$-smooth hypersurface $\Sigma \subset \dR^d$, $d\ge 2$, which is not necessarily bounded or closed, and the interaction 
strength is an arbitrary real valued bounded measurable function $\alpha$ 
on $\Sigma$. 
Following the approach of \cite{EI01, EK03} we show that the corresponding singular Schr\"odinger operator 
can be approximated in the norm resolvent sense by a family of regular ones with potentials suitably scaled in the direction perpendicular 
to $\Sigma$. 
We pay particular attention to the order of convergence and provide all preparatory technical integral estimates in a complete and 
self-contained form.  We shall also mention some spectral consequences of the general approximation result.

In the following we describe our main result.
Let $d \geq 2$ and let $\Sigma \subset \mathbb{R}^d$ be a bounded or unbounded orientable $C^2$-hypersurface as in Definition~\ref{definition_hypersurface},
and consider the symmetric sesquilinear form 
\begin{equation*}
  \mathfrak{a}_{\delta, \alpha}[f, g] = 
    \big( \nabla f, \nabla g \big)_{L^2(\mathbb{R}^d; \mathbb{C}^d)}
    - \int_{\Sigma}\alpha\, f|_\Sigma \,\overline{g|_\Sigma} \, \mathrm{d} \sigma,\qquad \dom \mathfrak{a}_{\delta, \alpha} = H^1(\mathbb{R}^d),
\end{equation*}
where $\alpha \in L^\infty(\Sigma)$ is a real valued function and $f|_\Sigma$, $g|_\Sigma$ denote the traces of functions $f,g \in H^1(\mathbb{R}^d)$ on~$\Sigma$.
Standard arguments yield that $\mathfrak{a}_{\delta, \alpha}$ is a densely defined, closed, and 
semibounded form in $L^2(\dR^d)$, and hence there exists a unique 
self-adjoint operator $A_{\delta, \alpha}$ in $L^2(\dR^d)$ such that
\begin{equation}\label{ada}
 (A_{\delta, \alpha}f,g)=\mathfrak{a}_{\delta, \alpha}[f, g],\qquad f\in\dom A_{\delta,\alpha},\,g\in\dom \mathfrak{a}_{\delta, \alpha},
\end{equation}
see Lemma~\ref{lemma_delta_form} for more details.
The operator $A_{\delta, \alpha}$ is regarded as a  Schr\"odinger operator 
    with a $\delta$-interaction of strength $\alpha$ supported on $\Sigma$ which corresponds to the formal singular differential expression 
    $-\Delta - \alpha \langle \delta_{\Sigma}, \cdot \rangle \delta_{\Sigma}$; cf. \cite{brasche_exner_kuperin_seba_1994} and \cite[Theorem 3.3]{BEL14_JPA}.
The choice of the negative potential sign is motivated
by the fact that interesting spectral features are in
this context usually associated with attractive
interactions.

Let $\nu$ be the continuous unit normal vector field on $\Sigma$, choose $\beta>0$ sufficiently small as in Hypothesis~\ref{hypothesis_hypersurface}
and consider layer neighborhoods $\Omega_\varepsilon$ of $\Sigma$ of the form
\begin{equation*}
\Omega_\varepsilon := \bigl\{x_\Sigma + t \nu(x_\Sigma):x_\Sigma\in\Sigma,\, t\in(-\varepsilon,\varepsilon) \bigr\},\qquad 0<\varepsilon\leq\beta.
\end{equation*}
Fix a real valued potential $V \in L^{\infty}(\mathbb{R}^d)$
with support in $\Omega_\beta$,  
define the scaled potentials $V_{\varepsilon} \in L^{\infty}(\mathbb{R}^d)$ with support in $\Omega_\varepsilon$ by
\begin{equation}\label{hohoho}
  V_{\varepsilon}(x) := 
  \begin{cases} \frac{\beta}{\varepsilon} V\big( x_{\Sigma} + 
  \frac{\beta}{\varepsilon} t \nu(x_{\Sigma}) \big), 
  &\text{if } x = x_{\Sigma} + t \nu(x_{\Sigma}) \in \Omega_{\varepsilon},
  \\ 0, &\text{else,} \end{cases}
\end{equation}
and consider the corresponding self-adjoint Schr\"odinger operators
\begin{equation*} 
  H_\varepsilon f = -\Delta f - V_\varepsilon f, 
  \qquad 
  \dom H_\varepsilon = H^2(\dR^d).
\end{equation*}

With these preparatory considerations we can formulate the main result of the
present paper.

\begin{thm} \label{theorem_convergence}
  Let $\Sigma \subset \mathbb{R}^d$, $d\geq 2$, be an orientable $C^2$-hypersurface as in Definition~\ref{definition_hypersurface}
  which satisfies Hypothesis~\ref{hypothesis_hypersurface}, let $Q \in L^{\infty}(\mathbb{R}^d)$ be real valued, and 
  let $V \in L^{\infty}(\mathbb{R}^d)$ be real valued with support in $\Omega_\beta$.
  Define $\alpha \in L^{\infty}(\Sigma)$ as the transversally averaged value of $t \mapsto V(x_\Sigma + t \nu(x_\Sigma))$ by
  \begin{equation*}
  \alpha(x_{\Sigma}) := 
  \int_{-\beta}^\beta V(x_{\Sigma} + s \nu(x_{\Sigma})) \mathrm{d} s
  \end{equation*}
  for a.e. $x_{\Sigma} \in \Sigma$ and let $A_{\delta, \alpha}$ be the corresponding 
  Schr\"{o}dinger operator with a $\delta$-interaction of strength $\alpha$ supported on $\Sigma$.
  Then there exists a $\lambda_0 < 0$ such that
  $(-\infty, \lambda_0) \subset 
  \rho(A_{\delta, \alpha} + Q) \cap \rho(H_\varepsilon + Q)$ for all $\varepsilon>0$ sufficiently small and
  for every $\lambda \in (-\infty, \lambda_0)$ there is a constant 
  $c = c(d, \lambda, \Sigma, V,Q) > 0$ such that 
  \begin{equation*}
    \left\| (H_\varepsilon +Q - \lambda)^{-1} - 
    (A_{\delta, \alpha} +Q - \lambda)^{-1} \right\| 
    \leq c\, \varepsilon \big( 1 + |\ln \varepsilon| \big)
  \end{equation*}
  holds for all $\varepsilon > 0$ sufficiently small.
  In particular, $H_\varepsilon+Q$ converges to 
  $A_{\delta, \alpha}+Q$ in the norm resolvent sense, as 
  $\varepsilon \rightarrow 0+$.
\end{thm}

Let us briefly describe the structure of the proof of Theorem~\ref{theorem_convergence} and the contents of this paper. 
Section~\ref{s2} contains preliminary material, definitions and properties of the hypersurfaces $\Sigma$ and their layer neighborhoods, as well as
a representation of the resolvent of the Schr\"{o}dinger operator $A_{\delta,\alpha}$ which goes back 
to~\cite{brasche_exner_kuperin_seba_1994}.
The heart of the proof of Theorem~\ref{theorem_convergence} is in Section~\ref{section_approximition}. The main part of this section deals with
the special case $Q = 0$. For this purpose, the potentials $V_\varepsilon$ in \eqref{hohoho} 
are factorized with the standard Birman-Schwinger method and useful representations of the resolvent of $H_\varepsilon$ are provided. 
The convergence analysis
is then done by comparing the different resolvent representations from Theorem~\ref{theorem_delta_op} and Proposition~\ref{theorem_resolvent_formula}, and essentially reduces
to convergence properties of certain integral operators discussed in Lemma~\ref{proposition_convergence}. However, the proof of
Lemma~\ref{proposition_convergence} requires various refined technical estimates for integrals containing the Green's function 
for the free Laplacian which are outsourced to Appendix~\ref{appa}.
We wish to mention that in Appendix~\ref{appa} particular attention is paid to keep the present paper self-contained. Therefore all necessary estimates are
presented in full detail and complete rigorous form; as a result Appendix~\ref{appa} is of mainly technical nature.
The statement of Theorem \ref{theorem_convergence} in the general case with $Q \neq 0$ follows then from the previous considerations
by a simple perturbation argument.
Eventually, there is a short Appendix~\ref{appendix_boundary}
in which it is shown that boundaries of bounded $C^2$-domains satisfy the assumptions imposed on the hypersurfaces $\Sigma$ in this paper.

Finally, we agree that
throughout the paper $c,C,C_k,\widetilde C_k$, $k \in \mathbb{N}$, denote constants that do not depend
on space variables and on $\varepsilon$. In the formulation of the results we usually write $C = C(\dots)$ 
to emphasize on which parameters these constants depend, but in the proofs we will mostly omit this.

{\bf Acknowledgements.} 
Jussi Behrndt and Vladimir Lotoreichik gratefully
acknowledge financial support
by the Austrian Science Fund (FWF): Project P~25162-N26.
Pavel Exner and Vladimir Lotoreichik also acknowledge financial support
by the Czech Science Foundation: Project 14-06818S.
Markus Holzmann thanks the Czech Centre for International Cooperation in Education (DZS) 
and the Austrian Agency for International Cooperation in Education and Research (OeAD)
for financial support under a scholarship of the program ``Aktion Austria - Czech Republic''
during a research stay in Prague.

\section{Preliminaries}\label{s2}

This section contains some preliminary material that will be 
useful in the main part of the paper. 
In Section~\ref{section_hypersurfaces_basics}
we recall certain basic facts from differential geometry
of hypersurfaces in Euclidean spaces 
and of layers built around these hypersurfaces.
Then, in Section \ref{section_delta_op}
we define Schr\"odinger operators with 
$\delta$-interactions supported on hypersurfaces
in a mathematically rigorous way.

\subsection{Hypersurfaces and their layer neighborhoods} \label{section_hypersurfaces_basics}

In this section we introduce several notions associated to 
hypersurfaces and layers around these hypersurfaces.
We follow the presentation from \cite{kuehnel_english}, 
which we adopt for our applications.
We start with a suitable definition of a class of hypersurfaces
in the Euclidean space $\dR^d$. We wish to emphasize that 
the hypersurfaces considered here are in general unbounded
and not necessarily closed; note also that the index set $I$ in the parametrization below is assumed to be finite.

\begin{definition} \label{definition_hypersurface}
We call $\Sigma \subset \dR^d$, $d \ge 2$, a $C^2$-hypersurface 
and $\{ \varphi_i, U_i, V_i \}_{i \in I}$ a parametrization of $\Sigma$, 
if $I$ is a finite index set and the following holds:
\begin{itemize}\skipex
	\item [\rm (a)] 
	$U_i \subset \dR^{d-1}$ 
	and $V_i \subset \dR^d$ 
	are open sets and  
	$\varphi_i\colon U_i \rightarrow V_i$ 
	is a $C^2$-mapping for all $i \in I$;
	\item[\rm (b)]
	 $\rank D \varphi_i(u) = d - 1$ 
	 for all $u \in U_i$ and $i\in I$;
    \item[\rm (c)] 
    $\varphi_i(U_i) = V_i \cap \Sigma$ 
    and $\varphi_i\colon U_i \rightarrow V_i \cap \Sigma$ 
    is a homeomorphism;
    \item[\rm (d)] $\Sigma \subset \bigcup_{i \in I} V_i$;
    \item[\rm (e)] 
    there exists a constant $C > 0$ such that
    \[
    |\varphi_i(u) - \varphi_i(v)| \le C |u - v|
    \] 
    for all $u,v \in U_i$ and $i\in I$.
    \end{itemize}
\end{definition}

Let $\Sigma \subset \dR^d$ be a $C^2$-hypersurface with 
parametrization $\{ \varphi_i, U_i, V_i \}_{i \in I}$. 
Then the inverse mappings $\varphi_i^{-1}: V_i\cap\Sigma \rightarrow U_i$ are often called charts 
and the family $\{ \varphi_i^{-1}, V_i\cap\Sigma, U_i \}_{i \in I}$ atlas of $\Sigma$.
For $x = \varphi_i(u)\in\Sigma$ ($u\in U_i$, $i\in I$) 
we denote the tangent hyperplane by 
\begin{equation*}
T_x := 
\spann 
\bigl\{ 
\partial_1 \varphi_i(u), \dots, \partial_{d-1} \varphi_i(u)    \bigr\}.
\end{equation*}
The tangent hyperplane $T_x$ is independent of the parametrization of $\Sigma$
and $\dim T_x = d-1$ holds by Definition~\ref{definition_hypersurface}~(b). Subsequently, it is assumed that 
$\Sigma$ is orientable, i.e. there exists a globally continuous unit normal vector field on $\Sigma$.
From now on we fix a continuous unit normal vector field (which is unique up to multiplication with $-1$) and denote it by 
$\nu(x)$ for $x\in\Sigma$. 
Then the mapping $U_i \ni u \mapsto \nu(\varphi_i(u))$ 
is continuously differentiable for all $i \in I$
and $\partial_j \nu(\varphi_i(u)) \in T_{\varphi_i(u)}$ 
for all $u \in U_i$ and $j \in \{1, \dots, d-1 \}$, 
see, e.g., \cite[Lemma 3.9 and Section 3F]{kuehnel_english}.

The first fundamental form $I_x$ 
associated to $\Sigma$ is the bilinear 
form on the tangent hyperplane $T_x$ defined by
\begin{equation*}
  I_x[a, b] := \langle a, b \rangle,\qquad   a, b \in T_x, 
\end{equation*}
where $\langle\cdot,\cdot\rangle$ denotes the standard scalar product in $\dR^d$.
For $x = \varphi_i(u)\in\Sigma$ ($u\in U_i$, $i\in I$) 
the matrix representing $I_x$ in the canonical basis 
$\{(\partial_j \varphi_i)(u)\}_{j=1}^{d-1}$ 
of $T_x$ is given by
\begin{equation}\label{Gi}
G_i(u) =\big(\big\langle (\partial_k \varphi_i)(u), 
    	(\partial_l \varphi_i)(u)\big\rangle \big)_{k, l = 1}^{d-1}
\end{equation} 
and also known as the metric tensor of $\Sigma$.
Observe that 
$G_i(u) = (D \varphi_i(u))^{\top} \cdot (D \varphi_i(u))$. Together with condition (b) in 
Definition~\ref{definition_hypersurface} this implies that
$G_i(u)$ is positive definite.

Finally, we introduce the notion of 
the Weingarten map or shape operator. 

\begin{definition} \label{definition_Weingarten_map}
Let $\Sigma \subset \dR^d$, $d\geq 2$, be an orientable $C^2$-hypersurface with parametrization $\{ \varphi_i, U_i, V_i \}_{i \in I}$
and let $\nu(x)$, $x\in\Sigma$, be a continuous unit normal vector field on $\Sigma$.  
For $x  = \varphi_i(u) \in \Sigma$ ($u \in U_i$, $i\in I$)
the Weingarten map  
$W(x)\colon T_x \rightarrow T_x$ 
is the linear operator acting on the basis vectors 
$\{\partial_j \varphi_i(u)\}_{j=1}^{d-1}$ of $T_x$ 
as 
\[
W(x) \partial_j \varphi_i(u) := -\partial_j \nu(\varphi_i(u)).
\]
\end{definition}

The Weingarten map $W(x)$ 
is well-defined (but its sign depends on the choice of the continuous unit normal vector field),
independent of the parametrization and symmetric with respect 
to the inner product induced by the first fundamental
form, see e.g. \cite[Lemma 3.9]{kuehnel_english} for the case $d=3$. 
For $x = \varphi_i(u)\in\Sigma$ ($u\in U_i$, $i\in I$)
the matrix associated to the linear mapping $W(x)$ 
corresponding to the canonical basis  
$\{\partial_j \varphi_i(u)\}_{j=1}^{d-1}$ of $T_x$ 
will be denoted by $L_i(u)$.

The eigenvalues $\{\mu_j(u)\}_{j=1}^{d-1}$
of $L_i(u)$ are the principal curvatures of $\Sigma$ and do not depend on the choice of the parametrization, 
see \cite[Definition 3.46]{kuehnel_english}. 
In particular, the quantity $\det(1 - t L_i(u))$ for $t \in \dR$,
which will appear later frequently, 
is independent of the parametrization 
and will also be denoted by
$\det(1 - t W(x))$. Furthermore, the eigenvalues of $W(x)$ depend continuously on
$x \in \Sigma$, as the entries of $L_i$ depend continuously on $u \in U_i$ (see the text after Definition~3.10 in \cite{kuehnel_english})
and $\varphi_i: U_i \rightarrow \Sigma \cap V_i$ is a homeomorphism.

Next, we discuss a convenient definition
of an integral for functions defined on the $C^2$-hypersurface $\Sigma$. For this fix a parametrization $\{\varphi_i, U_i, V_i \}_{i \in I}$ 
of $\Sigma$ as in Definition~\ref{definition_hypersurface} with a finite index set $I$ and choose a partition of unity subordinate to $\{ V_i \}_{i \in I}$, that is a family of 
functions $\chi_i: \mathbb{R}^d \rightarrow [0, 1]$, $i\in I$, with the following properties:
\begin{itemize}\skipex
 \item[\rm (i)] $\chi_i\in C^\infty(\mathbb{R}^d)$ for all $i \in I$;
 \item[\rm (ii)] $\supp \chi_i \subset V_i$
 for all $i \in I$;
 \item[\rm (iii)] 
   $\sum_{i \in I} \chi_i(x) = 1$ for any $x \in \Sigma$.
\end{itemize}
Note that some of the functions $\chi_i$, $i\in I$, are not compactly supported, if $\Sigma$ is unbounded.

A function $f\colon \Sigma \rightarrow \dC$ 
is said to be measurable (integrable), if 
\[
U_i \ni u \mapsto \chi_i(\varphi_i(u)) f(\varphi_i(u))
\]
is measurable (integrable, respectively) for all $i \in I$.
If $f\colon \Sigma \rightarrow \dC$ 
is integrable, we define the integral of $f$ over $\Sigma$ as
\begin{equation}
\label{integral}
\int_\Sigma f(x) \rmd \sigma(x) 
:= 
\sum_{i \in I}\int_{U_i} 
\chi_i(\varphi_i(u)) f(\varphi_i(u)) \sqrt{\det G_i(u)} \rmd u,
\end{equation}
where $\rmd u := \rmd \Lambda_{d-1}(u)$
denotes the usual $(d-1)$-dimensional Lebesgue measure
on $U_i$ and $G_i(u)$ is the matrix 
of the first fundamental form given in~\eqref{Gi}.
The measure $\sigma$ in \eqref{integral}
coincides with the canonical Hausdorff measure 
on $\Sigma$ which is independent of the parametrization of $\Sigma$; 
cf. \cite[Appendix~C.8]{leoni}.
Therefore, the above  definition of the integral does not depend
on the parametrization of $\Sigma$ and the choice of
the partition of unity. We denote the space of (equivalence classes of) square integrable 
functions $f:\Sigma\rightarrow\dC$ with respect to $\sigma$ by $L^2(\Sigma)$.

Next, we introduce layer neighborhoods of a $C^2$-hypersurface $\Sigma$ and we impose some additional 
conditions on $\Sigma$ in Hypothesis~\ref{hypothesis_hypersurface} below.
For this it is useful to define 
the functions
\begin{equation}
\label{iotavarphi}
\iota_{\varphi_i}\colon U_i \times \dR  \rightarrow \dR^d, 
    \qquad \iota_{\varphi_i}(u, t) :=  \varphi_i(u) + t \nu(\varphi_i(u)),\qquad i\in I.
\end{equation}
The Jacobian matrix of $\iota_{\varphi_i}$, $i\in I$, 
is given by the $d\times d$ matrix
\begin{equation*}
(D \iota_{\varphi_i})(u, t) 
= 
\begin{pmatrix}
(D \varphi_i)(u)
(1 - t L_i(u)) & \nu(\varphi_i(u)) 
\end{pmatrix}
\end{equation*}
and the absolute value of the determinant of this matrix
can be expressed as
\begin{equation}\label{DetDiota}   
\big|\det \big((D \iota_{\varphi_i})(u, t)\big)\big|
=  \big| \det(1 - t L_i(u)) \big| \sqrt{\det G_i(u)};
\end{equation}
cf. \cite[Section 2]{lin_lu07} and \cite[Section 3]{duclos_exner_krejcirik_2001}.
We will also make use of the mapping
\begin{equation}
\label{iota}
\iota_\Sigma \colon \Sigma\times \dR\rightarrow\dR^d,
\qquad 
\iota_\Sigma(x_\Sigma,t) := x_\Sigma + t \nu(x_\Sigma),
\end{equation}		
and layer neighborhoods
$\Omega_\beta$ of $\Sigma$ of the form
\begin{equation} 
\label{def_tube}
\Omega_\beta := \iota_\Sigma(\Sigma\times (-\beta,\beta)),\qquad \beta > 0.
\end{equation}

We employ the following hypothesis for the hypersurface $\Sigma$.

\begin{hypothesis} \label{hypothesis_hypersurface}
Let $\Sigma \subset \dR^d$ be an orientable $C^2$-hypersurface
with parametrization $\{\varphi_i, U_i, V_i\}_{i\in I}$. 
Assume that there exists $\beta > 0$ such that
  \begin{itemize}\skipex
    \item[\rm (a)] 
    the restriction of the mapping $\iota_\Sigma$ on $\Sigma\times (-\beta,\beta)$
    is injective;
   \item[\rm (b)] 
   	there is a constant $\eta\in (0,1)$ such that
    $\det(1-tW(x_\Sigma)) \in (1-\eta,1+\eta)$ for all $x_\Sigma \in \Sigma$
    and all $t\in (-\beta,\beta)$;
   \item[\rm (c)] 
    there exists a constant $c > 0$  
    such that the mappings
    $\iota_{\varphi_i}$ in \eqref{iotavarphi}
    satisfy
    \[
     \big|\iota_{\varphi_i}(u,t) - \iota_{\varphi_i}(v,s)\big|^2 
      \ge c^2\left( |u - v|^2 + |s -  t|^2 \right)
    \]
    for all $u, v \in U_i$, $s, t \in (-\beta, \beta)$ and all $i\in I$.
  \end{itemize}
\end{hypothesis}

All the assumptions of Hypothesis~\ref{hypothesis_hypersurface} are
satisfied for the boundary of a compact and
simply connected $C^2$-domain; see Appendix \ref{appendix_boundary}.
We also mention that a similar set of assumptions was imposed in \cite[Section~4]{brasche_exner_kuperin_seba_1994}.
In the next proposition it will be shown that item (c) in Hypothesis \ref{hypothesis_hypersurface} implies that
the eigenvalues of $W$ are uniformly bounded on $\Sigma$. 
In particular, this shows that (b) in Hypothesis~\ref{hypothesis_hypersurface}
is automatically satisfied if $\beta > 0$ is small enough.

\begin{prop} \label{proposition_eigenvalues_Weingarten_map}
Let $\Sigma \subset \dR^d$ be an orientable $C^2$-hypersurface and assume that item (c) in Hypothesis~\ref{hypothesis_hypersurface}
holds. Then the eigenvalues of
the matrix of the Weingarten map are uniformly bounded on $\Sigma$.
\end{prop}

\begin{proof}
Let $\beta>0$ be as in Hypothesis \ref{hypothesis_hypersurface} (c) and suppose
that the eigenvalues of the Weingarten map are not uniformly bounded.
Then for some $i\in I$ there exists $u \in U_i$ and an eigenvalue $\mu$ of $L_i(u)$ such that $\vert\mu\vert> \beta^{-1}$.
Choose a sequence $(s_n) \subset (-\beta, \beta)$
such that $s_n^{-1}$ are not
eigenvalues of $L_i(u)$ and $s_n \rightarrow \mu^{-1}$. Then
\begin{equation*}
 \det(1 - s_n L_i(u))\neq 0\quad\text{and}\quad \det(1 - s_n L_i(u))\rightarrow 0
\end{equation*}
and as $G_i(u)$ is positive definite and has uniformly bounded values by Definition \ref{definition_hypersurface} (e),
the same holds for $\det(1 - s_n L_i(u))\sqrt{\det G_i(u)}$, that is,
\begin{equation*}
 \det D \iota_{\varphi_i}(u, s_n)\neq 0\quad\text{and}\quad \det D \iota_{\varphi_i}(u, s_n)\rightarrow 0;
\end{equation*}
cf. \eqref{DetDiota}.
From $D \iota_{\varphi_i}^{-1} (\iota_{\varphi_i}(u,s_n))=(D \iota_{\varphi_i}(u, s_n))^{-1}$ we conclude
 \begin{equation}\label{oho}
    \det D \iota_{\varphi_i}^{-1} (\iota_{\varphi_i}(u,s_n))
      =  \frac{1}{\det D \iota_{\varphi_i}(u, s_n)}  \rightarrow \infty.
  \end{equation}
On the other hand, by Hypothesis \ref{hypothesis_hypersurface} (c) the mapping $\iota_{\varphi_i}^{-1}$ is
Lipschitz continuous on $\iota_{\varphi_i}(U_i \times (-\beta, \beta))$ and hence $\| D \iota_{\varphi_i}^{-1} \|$ is bounded;
this contradicts \eqref{oho}.
\end{proof} 

In the next example we provide a $C^2$-hypersurface which does not satisfy Hypothesis \ref{hypothesis_hypersurface};
here a curve in~$\dR^2$ with unbounded curvature at ``infinity'' is discussed.

\begin{example}
  \rm Consider the curve 
  \begin{equation*}
    \varphi\colon \dR \rightarrow \dR^2, 
    \quad 
    u \mapsto \begin{pmatrix} u \\ \int_0^u \sin(t^2) \rmd t 
    \end{pmatrix},
  \end{equation*}
  and observe that $\varphi(\mathbb{R})$ 
  is an orientable $C^2$-hypersurface in $\dR^2$ with parametrization $\{ \varphi, \dR, \dR^2 \}$.
  If we fix the unit normal vector field by 
  \begin{equation*}
  \nu(u)=\frac{1}{(1 + \sin^2(u^2))^{1/2}}\begin{pmatrix} -\sin (u^2) \\ 1\end{pmatrix},
  \end{equation*}
  then the corresponding  $1 \times 1$-matrix of the Weingarten map 
  is given by
  \begin{equation*}
    L(u) = \frac{2 u \cos(u^2)}{(1 + \sin^2(u^2))^{3/2}}.
  \end{equation*}
  Clearly, $L$ is unbounded and hence, item {\rm (b)} 
  in Hypothesis~\ref{hypothesis_hypersurface} is not satisfied.
\end{example}

Under Hypothesis~\ref{hypothesis_hypersurface}, 
the mapping $\iota_\Sigma$ in~\eqref{iota}
is bijective from $\Sigma\times (-\beta,\beta)$
onto $\Omega_\beta$. 
This allows us to identify functions $f$ supported 
on $\Omega_\beta$ with functions $\wt f$ defined on 
$\Sigma \times (-\beta, \beta)$ 
via the natural identification
\begin{equation*}
 f(x) = f(\iota_\Sigma(x_\Sigma, t)) = \wt f(x_\Sigma, t),\quad x
 = \iota_\Sigma(x_\Sigma, t),
\quad x_\Sigma\in \Sigma,\,t\in(-\beta,\beta).
\end{equation*}
Subsequently, $L^1(\Omega_\beta)$ is equipped with the $d$-dimensional Lebesgue measure $\Lambda_d$ 
and $L^1(\Sigma\times (-\beta,\beta))$ is equipped with the measure $\sigma \times \Lambda_1$. 
In the next proposition it is shown that $L^1(\Omega_\beta)$ and $L^1(\Sigma\times (-\beta,\beta))$ can be identified and a 
useful change of variables formula is provided.

\begin{prop} \label{corollary_integration_tube}
Let $\Sigma \subset \dR^d$ be an orientable $C^2$-hypersurface, assume that Hypothesis~\ref{hypothesis_hypersurface} is satisfied and
 let $\Omega_\beta$ be as in~\eqref{def_tube}. 
Then the following assertions are true.
\begin{itemize}\skipex
\item [\rm (i)] 
Let $\iota_\beta := \iota_\Sigma|_{\Sigma \times (-\beta, \beta)}$.
Then, there exist constants $0 < c_1\le c_2 < +\infty$
such that 
\[
c_1\|f\|_{L^1(\Omega_\beta)}
\le 
\| f \circ \iota_\beta \|_{L^1(\Sigma\times (-\beta,\beta))}
\le
c_2\|f\|_{L^1(\Omega_\beta)},\qquad f\in L^1(\Omega_\beta).
\]
In particular, 
$f \in L^1(\Omega_\beta)$ 
if and only if
$f \circ \iota_\beta \in L^1(\Sigma\times (-\beta,\beta))$.

\item [\rm (ii)] 
For $f\in L^1(\Omega_\beta)$ 
the identity
\begin{equation*}
    \int_{\Omega_\beta} f(x) \rmd x 
    =\int_{\Sigma} \int_{-\beta}^{\beta} 
    f(x_\Sigma + t \nu(x_\Sigma)) \det(1 - t W({x_\Sigma})) \rmd t \rmd \sigma(x_\Sigma) 
\end{equation*}
holds, where $W$ is the Weingarten map associated to $\Sigma$.
\end{itemize}
\end{prop}

\begin{proof}
Let $\{ \varphi_i, U_i, V_i \}_{i \in I}$ 
be a parametrization of $\Sigma$ with a finite index set $I$,
let $\{ \chi_i \}_{i \in I}$ be a partition of unity 
subordinate to the covering $\{ V_i \}_{i \in I}$ and set
\[
\wt\chi_i(x) := \chi_i(x_\Sigma),\qquad i\in I,
\]
for $x = \iota_\beta(x_\Sigma, t) = x_\Sigma+ t\nu(x_\Sigma)\in \Omega_\beta$ with $x_\Sigma\in \Sigma$
and $t\in (-\beta,\beta)$.
The family $\{\wt\chi_i\}_{i\in I}$ satisfies 
\begin{equation}
\label{property_wt_chi}
\sum_{i\in I}\wt\chi_i(x) =1\quad\text{for all}~x\in \Omega_\beta
\end{equation}
due to the properties of the partition of unity 
$\{ \chi_i \}_{i \in I}$. 

Let 
$f\in L^1(\Omega_{\beta})$, let $\iota_{\varphi_i}$ be as in~\eqref{iotavarphi} and
let $\Omega_{i,\beta} := \iota_{\varphi_i}(U_i\times (-\beta,\beta))$. 
Using~\eqref{property_wt_chi} we get
\[
\int_{\Omega_\beta} f(x)\rmd x 
= 
\int_{\Omega_\beta} \sum_{i\in I} \wt\chi_i(x) f(x)\rmd x 
= 
\sum_{i\in I}\int_{\Omega_{i,\beta}} \wt\chi_i(x) f(x)\rmd x.
\]
Making the substitution $x = \iota_{\varphi_i}(u,t)$, $i\in I$,
in  each summand of the last formula,
we get with the aid of~\eqref{DetDiota} 
\[
\int_{\Omega_\beta} f(x)\rmd x =\sum_{i\in I}\int_{U_i} 
\int_{-\beta}^\beta
\wt\chi_i(\iota_{\varphi_i}(u,t)) 
f(\iota_{\varphi_i}(u,t))
\det(1 - tL_i(u))\sqrt{\det G_i(u)}
\rmd t \rmd u.
\]
Using $\wt\chi_i(\iota_{\varphi_i}(u,t))=\chi_i(\varphi_i(u))$, \eqref{integral} 
and $\iota_{\varphi_i}(u, t) = \iota_\beta(\varphi_i(u), t)$, we end up with
\[
\int_{\Omega_\beta} f(x)\rmd x  
= \int_{-\beta}^\beta\int_{\Sigma} f(\iota_\beta(x_\Sigma,t))
\det(1 - tW({x_\Sigma}))\rmd\sigma(x_\Sigma)\rmd t.
\]
Together with Hypothesis~\ref{hypothesis_hypersurface}~{\rm (b)}
this formula implies assertions (i) and (ii).
\end{proof}

\subsection{Schr\"{o}dinger operators with $\delta$-interactions on hypersurfaces} \label{section_delta_op}

In this section we recall a representation for the resolvent of the self-adjoint
Schr\"odinger operator $A_{\delta, \alpha}$ in $L^2(\mathbb{R}^d)$ 
with a $\delta$-interaction supported on the $C^2$-hypersurface $\Sigma$. 
As in \cite{brasche_exner_kuperin_seba_1994} the operator $A_{\delta, \alpha}$ is defined via the corresponding quadratic form
with the help of the first representation theorem \cite[Theorem VI 2.1]{kato};
the functions in the domain of $A_{\delta, \alpha}$ then satisfy the typical 
$\delta$-type boundary conditions on $\Sigma$, see, e.g.,
\cite[Theorem 3.3]{BEL14_JPA}. In the following $A_{\delta, \alpha}$ is the unique self-adjoint operator in $L^2(\dR^d)$ 
associated to the form $\mathfrak{a}_{\delta, \alpha}$ in the next lemma; cf. \eqref{ada}.

\begin{lem} \label{lemma_delta_form}
Let $\Sigma \subset \dR^d$ be a $C^2$-hypersurface, assume that Hypothesis~\ref{hypothesis_hypersurface} is satisfied and
let $\alpha \in L^\infty(\Sigma)$ be a real valued function. Then the symmetric sesquilinear form 
\begin{equation} \label{def_delta_form}
  \mathfrak{a}_{\delta, \alpha}[f, g] := 
    \big( \nabla f, \nabla g \big)_{L^2(\mathbb{R}^d; \mathbb{C}^d)}
    - \int_{\Sigma}\alpha\, f|_\Sigma \,\overline{g|_\Sigma} \, \mathrm{d} \sigma,\quad \dom \mathfrak{a}_{\delta, \alpha} := H^1(\mathbb{R}^d),
\end{equation}
is densely defined, closed and bounded from below in $L^2(\dR^d)$;
here $f|_\Sigma$, $g|_\Sigma$ denote the traces of functions $f,g \in H^1(\mathbb{R}^d)$ on $\Sigma$. 
\end{lem}

\begin{proof}
First, we note that Lemma \ref{lemma1} (i) together with 
\cite[Theorem VII 2, Remark VI 1]{jonsson_wallin_book} 
(see also \cite[Section VIII 1.1]{jonsson_wallin_book}) imply
that 
for $\tfrac{1}{2}<s\leq 1$ there is a bounded trace operator 
from $H^s(\mathbb{R}^d)$ to $L^2(\Sigma)$. In particular, the form $\mathfrak{a}_{\delta, \alpha}$ in \eqref{def_delta_form} is well-defined. Moreover, since
$H^1(\dR^d)$ is dense in $L^2(\dR^d)$, the form $\fra_{\delta,\alpha}$ is densely defined
in $L^2(\dR^d)$.

Subsequently, 
fix some $s\in\big(\tfrac{1}{2},1\big)$ and $c_s>0$ such that $\Vert f\vert_\Sigma\Vert_{L^2(\Sigma)}\leq \sqrt{c_s}\Vert f\Vert_{H^s(\dR^d)}$ for all 
$f\in H^1(\dR^d)$.
Let $\varepsilon>0$ and use \cite[Theorem 3.30]{haroske_triebel_2008} or 
\cite[Satz 11.18 (e)]{weidmann1}
to see that there exists a $C(\varepsilon)>0$ such that
\begin{equation}\label{estimate}
\begin{split}
 \vert (\alpha f\vert_\Sigma,f\vert_\Sigma)_{L^2(\Sigma)}\vert&\leq\Vert \alpha\Vert_\infty \Vert f\vert_\Sigma\Vert_{L^2(\Sigma)}^2 
 \leq c_s \Vert \alpha\Vert_\infty \Vert f\Vert_{H^s(\dR^d)}^2\\
 &\leq c_s \Vert \alpha\Vert_\infty \left(\varepsilon \Vert f\Vert_{H^1(\dR^d)}^2+C(\varepsilon)\Vert f\Vert_{L^2(\dR^d)}^2\right).
 \end{split}
\end{equation}
Thus, for sufficiently small $\varepsilon >0$ the form $f \mapsto (\alpha f\vert_\Sigma,f\vert_\Sigma)$ on $H^1(\dR^d)$ is relatively bounded with respect to the 
closed and nonnegative form 
$f \mapsto ( \nabla f, \nabla f)_{L^2(\dR^d; \dC^d)}$ on $H^1(\dR^d)$ with 
 bound smaller than one. Then by \cite[Theorem~VI~1.33]{kato} the form 
$\mathfrak{a}_{\delta, \alpha}$ in \eqref{def_delta_form} is closed and bounded from below.  
\end{proof}

Next, we provide a formula for the resolvent of the Schr\"{o}dinger operator $A_{\delta, \alpha}$.
For this purpose some notations are required. The free Laplace operator in $L^2(\mathbb{R}^d)$
with domain $H^2(\mathbb{R}^d)$ is denoted by $-\Delta$; it is clear that $-\Delta$ coincides with $A_{\delta,0}$ ($\alpha \equiv 0$) in the above lemma.
The spectrum of $-\Delta$ is given by
$\sigma(-\Delta) = [0, \infty)$. For $\lambda \in (-\infty, 0) \subset \rho(-\Delta)$
we define the function
\begin{equation} \label{def_G_lambda}
  G_{\lambda}(x) 
  := \frac{1}{(2 \pi)^{d/2}} \left( \frac{|x|}{\sqrt{-\lambda}} \right)^{1 - d/2} 
  K_{d/2 - 1} \left( \sqrt{-\lambda} |x| \right), 
  \quad x \in \mathbb{R}^d \setminus \{ 0 \},
\end{equation}
where $K_{d/2 - 1}$ denotes a modified Bessel 
function of the second kind and order $\frac{d}{2} - 1$, see \cite{abramowitz_stegun} for the 
definition and the properties of these functions.
Then, 
\begin{equation*}
  (R(\lambda) f)(x) := \big((-\Delta - \lambda)^{-1} f\big)(x) 
  = \int_{\mathbb{R}^d} G_{\lambda} (x - y) f(y) \mathrm{d} y;
\end{equation*}
cf. \cite[Section 7.4]{teschl_book}.
Next we define for $\lambda \in (-\infty, 0)$
integral operators $\gamma(\lambda)$, $M(\lambda)$
and provide the integral representation for the adjoint of  $\gamma(\lambda)$
\begin{subequations}\label{def_gamma_field_Weyl_function}
\begin{align} 
	\label{def_gamma_field} 
    &\!\gamma(\lambda)\colon\!  L^2(\Sigma) \rightarrow L^2(\mathbb{R}^d),&\quad\!\!
    (\gamma(\lambda) \xi)(x) &\!:=\! \int_{\Sigma} G_{\lambda}(x - y_{\Sigma}) \xi(y_\Sigma) 
    \mathrm{d} \sigma(y_\Sigma); \\
    \label{def_Weyl_function} 
    &\!M(\lambda)\colon\! L^2(\Sigma) \rightarrow L^2(\Sigma), &\quad\!\!
    (M(\lambda) \xi)(x_\Sigma)&\! :=\! \int_{\Sigma} G_{\lambda}(x_\Sigma - y_{\Sigma}) 
    \xi(y_\Sigma) \mathrm{d} \sigma(y_\Sigma); \\
    \label{def_gamma_field_adjoint} 
    &\!\gamma(\lambda)^*\colon\! L^2(\mathbb{R}^d) \rightarrow L^2(\Sigma),& \quad\!\!
    (\gamma(\lambda)^* f)(x_\Sigma) &=\! \int_{\mathbb{R}^d} G_{\lambda}(x_\Sigma - y) f(y) 
    \mathrm{d} y.
\end{align}
\end{subequations}
For our later considerations the resolvent formula in the next theorem is particularly useful. In the proof of item (a) and later in Section~\ref{section_approximition}
the Schur test for integral operators will be used frequently, see, e.g., \cite[Example III 2.4]{kato} or \cite[Satz 6.9]{weidmann1}.

\begin{thm} \label{theorem_delta_op}
Let $\Sigma \subset \dR^d$ be a $C^2$-hypersurface which satisfies Hypothesis~\ref{hypothesis_hypersurface} and
let $\alpha \in L^\infty(\Sigma)$ be a real valued function. Then the following statements are true.
  \begin{itemize}
    \item[\rm (a)] For $\lambda \in (-\infty, 0)$ 
    the operators $\gamma(\lambda), M(\lambda)$ and $\gamma(\lambda)^*$ in
    \eqref{def_gamma_field_Weyl_function}
    are bounded and everywhere defined.     
    \item[\rm (b)] There exists a $\lambda_0 < 0$
    such that $1 - \alpha M(\lambda)$ admits a bounded and everywhere defined
    inverse for all $\lambda \in (-\infty, \lambda_0)$. 
    These $\lambda$ belong to $\rho(A_{\delta, \alpha})$ and it holds
    \begin{equation*}
      (A_{\delta, \alpha} - \lambda)^{-1} 
        = R(\lambda) + \gamma(\lambda) \big(1 - \alpha M(\lambda)\big)^{-1} 
        \alpha \gamma(\lambda)^*.
    \end{equation*}
  \end{itemize}
\end{thm}

\begin{proof}
 (a) Let $\lambda \in (-\infty, 0)$. In order to prove that $\gamma(\lambda)$
  is well-defined and bounded we use the Schur test. In fact, from 
  Proposition \ref{proposition2} (i) and Proposition \ref{proposition_appendix_hat_1} (i) we obtain
  \begin{equation*}
    \begin{split}
      \| \gamma(\lambda) \|^2 &\leq \sup_{x \in \mathbb{R}^d} \int_{\Sigma} \big| G_{\lambda}(x - y_{\Sigma})\big| \mathrm{d} \sigma(y_{\Sigma})
         \cdot \sup_{y_{\Sigma} \in \Sigma} \int_{\mathbb{R}^d} \big| G_{\lambda}(x - y_{\Sigma}) \big| \mathrm{d} x < \infty.
    \end{split}
  \end{equation*}
  In a similar way one can show that $M(\lambda)$ and $\gamma(\lambda)^*$ are bounded. 
  
  Item (b) is essentially a variant of \cite[Lemma 2.3]{brasche_exner_kuperin_seba_1994}. 
  In fact, let us define for Borel sets $B \subset \mathbb{R}^d$ the measure $m$ by
\begin{equation} \label{measure_Kato}
  m(B) := \sigma(B \cap \Sigma),
\end{equation}
where 
$\sigma$ is the measure in \eqref{integral}. Then $m(\mathbb{R}^d \setminus \Sigma) = 0$,
the spaces $L^2(\mathbb{R}^d; m)$ and $L^2(\Sigma)$ can be identified and for $f\in L^1(\Sigma)$
one has 
\begin{equation*}
 \int_\Sigma f(x_\Sigma) \rmd\sigma(x_\Sigma) = \int_{\dR^d} \widetilde f(x)\,\rmd m(x), 
\end{equation*}
where $\widetilde f$ is some extension of $f$ onto the $m$-null set $\mathbb{R}^d \setminus \Sigma$.
Moreover, the estimate \eqref{estimate} shows that the measure $m$ in \eqref{measure_Kato} satisfies \cite[eq. (2.1)]{brasche_exner_kuperin_seba_1994} 
(with $\gamma=-\alpha$).
Now it is easy to see that the integral operators $\gamma(\lambda)$, $M(\lambda)$, and
$\gamma(\lambda)^*$ in \eqref{def_gamma_field_Weyl_function} can be identified with the
operators $R_{m\, \rmd x}(i \sqrt{-\lambda})$, $R_{m m}(i \sqrt{-\lambda})$ and $R_{\rmd x\, m}(i \sqrt{-\lambda})$ in \cite{brasche_exner_kuperin_seba_1994}, respectively.
The assertion in item (b) follows from \cite[Lemma 2.3 (ii) and~(iii)]{brasche_exner_kuperin_seba_1994}.
\end{proof}

\section{Approximation of $A_{\delta, \alpha}$ by Schr\"odinger operators with regular potentials} \label{section_approximition}

In this section we prove Theorem~\ref{theorem_convergence}, the main result of this paper.
First, in Section \ref{subsection_H_eps} we recall briefly the definitions of the layer neigborhoods, the scaled potentials and the associated 
Hamiltonians $H_\varepsilon$ from the Introduction, and we derive a resolvent formula for $H_\varepsilon$
which is convenient in the convergence analysis.
Section \ref{subsection_convergence} contains the main part of the proof of Theorem \ref{theorem_convergence}.
It is efficient to prove Theorem~\ref{theorem_convergence} for the special case $Q=0$ first; 
all technical estimates in Appendix~\ref{appa} and all 
preparatory steps  in Sections~\ref{subsection_H_eps} and~\ref{subsection_convergence} are tailormade for this case. 
The general case $Q\not=0$
is treated with a simple perturbation argument in the last step of the proof of Theorem~\ref{theorem_convergence}.
Finally, 
in Section~\ref{subsection_consequences} we discuss some connections
between the spectral properties of $H_\varepsilon$ and of $A_{\delta,\alpha}$
that follow from Theorem~\ref{theorem_convergence}.

\subsection{Preliminary considerations on $H_\varepsilon$} \label{subsection_H_eps}

Let $d \geq 2$ and $\Sigma \subset \mathbb{R}^d$ be a $C^2$-hypersurface
which satisfies Hypothesis~\ref{hypothesis_hypersurface}.
Then the mapping 
$\iota_\Sigma$ in \eqref{iota} is injective on $\Sigma \times (-\beta, \beta)$ 
for some (in the following fixed) $\beta >0$ as in 
Hypothesis~\ref{hypothesis_hypersurface}.
Recall the definition of the layer $\Omega_\varepsilon$
from \eqref{def_tube} for $\varepsilon \in (0, \beta]$,
fix a real valued potential $V \in L^{\infty}(\mathbb{R}^d)$
with $\supp V \subset \Omega_\beta$ and consider
the scaled potentials 
\begin{equation*}
  V_{\varepsilon}(x) = 
  \begin{cases} \frac{\beta}{\varepsilon} V\left( x_{\Sigma} + 
  \frac{\beta}{\varepsilon} t \nu(x_{\Sigma}) \right), 
  &\text{if } x = x_{\Sigma} + t \nu(x_{\Sigma}) \in \Omega_{\varepsilon},
  \\ 0, &\text{else,} \end{cases}
\end{equation*}
where $\nu$ is
the continuous unit normal vector field on $\Sigma$.
Observe that $V_\beta = V$ and $\supp V_\varepsilon \subset \Omega_\varepsilon$.
The associated self-adjoint Schr\"odinger operators are given by
\begin{equation} \label{def_H_eps}
  H_\varepsilon f = -\Delta f - V_{\varepsilon} f, 
  \qquad 
  \dom H_\varepsilon = H^2(\dR^d).
\end{equation}

Our main objective in this section is to derive the resolvent formula for $H_\varepsilon$ in Proposition~\ref{theorem_resolvent_formula}
which turns out to be particularly convenient for our convergence analysis.
We start with the standard factorization of the potentials $V_\varepsilon= v_\varepsilon u_\varepsilon$, where 
\begin{equation} \label{def_u_eps}
  u_\varepsilon: L^2(\mathbb{R}^d) \rightarrow L^2(\Omega_\varepsilon), \quad 
  (u_{\varepsilon} f)(x) := |V_{\varepsilon}(x)|^{1/2} f(x), \quad x \in \Omega_\varepsilon,
\end{equation}
and 
\begin{equation} \label{def_v_eps}
  v_\varepsilon: L^2(\Omega_\varepsilon) \rightarrow L^2(\mathbb{R}^d), \quad 
  (v_{\varepsilon} h)(x) := \begin{cases} \mathrm{sign}\, V_\varepsilon(x) |V_{\varepsilon}(x)|^{1/2} h(x), &x \in \Omega_\varepsilon, \\
                                       0, &\text{else.} \end{cases}
\end{equation}
Recall that for $\lambda \in \rho(-\Delta) = \mathbb{C} \setminus [0, \infty)$ 
the resolvent of $-\Delta$ is denoted by $R(\lambda) = (-\Delta - \lambda)^{-1}$.
The following proposition contains a first auxiliary resolvent formula for $H_\varepsilon$.

\begin{prop} \label{prop_first_resolvent_formula}
  Let $H_\varepsilon$ be defined as in \eqref{def_H_eps} and let 
  $u_{\varepsilon}, v_{\varepsilon}$ and $R(\lambda)$ 
  be given as above. Then the following assertions are true.
  \begin{itemize}
  \item[\rm (i)] For all $\lambda \in \mathbb{C} \setminus [0, \infty)$ 
  with $1\in\rho(u_{\varepsilon} R(\lambda) v_{\varepsilon})$
  one has $\lambda \in \rho(H_\varepsilon)$ and
  \begin{equation*}
      (H_\varepsilon - \lambda)^{-1} = R(\lambda) + 
      R(\lambda) v_{\varepsilon} \left( 1 - u_{\varepsilon}R(\lambda)
            v_{\varepsilon} \right)^{-1} u_{\varepsilon}R(\lambda). 
  \end{equation*}
  
  \item[\rm (ii)] For all $M \in (0, 1)$ there exists $\lambda_M < 0$ such that
  \begin{equation*}
    \left\| u_{\varepsilon} R(\lambda) v_{\varepsilon} \right\| \leq M
  \end{equation*}
  holds for all $\varepsilon \in (0, \beta]$ and $\lambda < \lambda_M$.
  In particular, for these $\lambda$ the results from {\rm (i)} apply 
  and hence $(-\infty, \lambda_M) \subset \rho(H_\varepsilon)$
  for all $\varepsilon \in (0, \beta]$.
  \end{itemize}
\end{prop}
\begin{proof}

  (i) Let $\lambda \in \mathbb{C} \setminus [0, \infty)$  be such that
  $ 1 \in \rho(u_{\varepsilon} R(\lambda) v_{\varepsilon})$. 
  Note that $\lambda$ is not an eigenvalue of $H_\varepsilon$, as
  otherwise 
  $1\in\sigma_{\rm p}(u_{\varepsilon} R(\lambda) v_{\varepsilon})$;
  cf. \cite[Lemma 1]{brasche1995}.
   Next, we define the operator
  \begin{equation*}
    T(\lambda) := 
    R(\lambda) + R(\lambda) v_{\varepsilon} 
    \left( 1 - u_{\varepsilon} R(\lambda) v_{\varepsilon} \right)^{-1} 
    u_{\varepsilon} R(\lambda).
  \end{equation*}
  This operator is well defined and bounded,   
  as $1 \in \rho(u_{\varepsilon} R(\lambda) v_{\varepsilon})$.
  Using $V_\varepsilon = v_\varepsilon u_\varepsilon$ we conclude that
  \begin{equation*}
    \begin{split}
      (H_\varepsilon - \lambda)T(\lambda) f 
       &= \big(-\Delta - \lambda - v_{\varepsilon} u_{\varepsilon}\big) 
           T(\lambda) f \\
       &= f + v_{\varepsilon} \left( 1 - u_{\varepsilon} R(\lambda) v_{\varepsilon} \right)^{-1}
           u_{\varepsilon} R(\lambda) f - v_{\varepsilon} u_{\varepsilon} R(\lambda) f \\
      &\quad \quad - v_{\varepsilon} (1 - 1 + u_{\varepsilon} R(\lambda) v_{\varepsilon}) 
           \left( 1 - u_{\varepsilon} R(\lambda) v_{\varepsilon} \right)^{-1} 
           u_{\varepsilon} R(\lambda) f \\
      &= f + v_{\varepsilon} \left( 1 - u_{\varepsilon} R(\lambda) v_{\varepsilon} \right)^{-1} 
           u_{\varepsilon} R(\lambda) f - v_{\varepsilon} u_{\varepsilon} R(\lambda) f \\
      &\quad \quad - v_{\varepsilon} \left( 1 - u_{\varepsilon} R(\lambda) v_{\varepsilon} \right)^{-1} 
           u_{\varepsilon} R(\lambda) f + v_{\varepsilon} u_{\varepsilon} R(\lambda) f 
      = f
    \end{split}
  \end{equation*}
  holds for any $f \in L^2(\mathbb{R}^d)$.
  Hence, $(H_\varepsilon - \lambda)$ is bijective,
  which implies that $\lambda\in\rho(H_\varepsilon)$, and 
  \[
	  (H_\varepsilon- \lambda)^{-1} =  T(\lambda)=  
    R(\lambda) + R(\lambda) v_{\varepsilon} 
    \left( 1 - u_{\varepsilon} R(\lambda) v_{\varepsilon} \right)^{-1} 
    u_{\varepsilon} R(\lambda).
  \]
  This proves assertion (i).

  (ii) Let $\lambda \in (-\infty, 0)$ and recall that $R(\lambda)$ can be expressed by
  $R(\lambda) f = \int_{\mathbb{R}^d} G_{\lambda}(\cdot - y) f(y) \mathrm{d} y$
  with $G_\lambda$ as in~\eqref{def_G_lambda}.
  Let $M \in (0, 1)$ be fixed. Using the Schur test 
and that the absolute value
of the integral kernel of 
$u_\eps R(\lm) v_\eps$ is symmetric, 
  we find that
  \begin{equation*}
      \big\| 
      u_{\varepsilon} R(\lm) v_{\varepsilon} \big\|
      \leq \sup_{x \in \Omega_\varepsilon} \int_{\Omega_\varepsilon} |V_\varepsilon(x)|^{1/2}
      |G_{\lambda}(x - y)| |V_\varepsilon(y)|^{1/2}
      \dd y
      \leq \frac{\beta}{\varepsilon} \|V\|_{L^\infty} \sup_{x \in \Omega_\varepsilon}  \int_{\Omega_\varepsilon} 
      |G_{\lambda}(x - y)| 
      \dd y.
 \end{equation*} 
  Hence, the claimed result follows from Proposition \ref{proposition2} (ii),
  as this shows the existence of a number $\lambda_M < 0$ such that
  \begin{equation*} 
   \frac{\beta}{\varepsilon} \|V\|_{L^\infty} 
    \int_{\Omega_{\varepsilon}} |G_{\lambda}(x - y)| \mathrm{d} y \leq M
  \end{equation*}
  for any $x \in \Omega_\varepsilon$, all $\lambda < \lambda_M$ and any $\varepsilon \in (0, \beta]$.
\end{proof}

Next, we transform the resolvent formula from Proposition 
\ref{prop_first_resolvent_formula} into another one, 
which is more convenient for the convergence analysis.
This requires several preparatory steps . Recall that for an interval
$I \subset \mathbb{R}$ the space $L^2(\Sigma \times I)$
is equipped with the product measure $\sigma \times \Lambda_1$
of the Hausdorff measure on $\Sigma$ and the one-dimensional Lebesgue measure.
Define the functions 
$u, v \in L^{\infty}(\Sigma \times (-1, 1))$ by
\begin{equation} \label{def_u_v}
  u(x_{\Sigma}, t) := |\beta V(x_{\Sigma} + \beta t \nu(x_{\Sigma}))|^{1/2} 
  \quad \text{and} \quad 
  v(x_{\Sigma}, t) := \sign \big(V(x_{\Sigma} + \beta t \nu(x_{\Sigma}))\big) \cdot u(x_{\Sigma}, t).
\end{equation}
The following operators are essential to
state a convenient resolvent formula for $H_\varepsilon$.
We define for $\eps \in [0,\beta]$ and $\lambda \in (-\infty, 0)$ the integral operators
$A_\eps(\lm)\colon 	L^2(\Sigma \times (-1, 1)) \arr L^2(\dR^d)$,
$B_\eps(\lm)\colon L^2(\Sigma \times (-1, 1)) \arr 
	L^2(\Sigma \times (-1, 1))$ and
$C_\eps(\lm)\colon L^2(\dR^d) \arr L^2(\Sigma \times (-1, 1))$
as
\begin{subequations} \label{def_ABC_eps}
\begin{align}
  \label{def_A_eps}
    (A_\eps(\lm) \Xi)(x) 
    &:= 
    \int_\Sigma \int_{-1}^1 
    G_\lm(x - y_\Sigma - \eps s \nu(y_{\Sigma}))v(y_{\Sigma}, s)
    \det(1 - \eps s W(y_{\Sigma})) \Xi(y_{\Sigma}, s) 
    \dd s \dd \sigma(y_{\Sigma});\\
  \label{def_B_eps}
	(B_\eps(\lm) \Xi)(x_{\Sigma}, t) 
    &:= 
    u(x_{\Sigma}, t) 
    \int_{\Sigma} \int_{-1}^1 
    G_\lm(x_{\Sigma} + \eps t \nu(x_\Sigma) - y_\Sigma - 
    \eps s \nu(y_{\Sigma})) \\
  \notag
    & \qquad \qquad \qquad \qquad \qquad \qquad \qquad \cdot 
   	v(y_{\Sigma}, s) \det(1 - \eps s W(y_{\Sigma})) \Xi(y_{\Sigma}, s) 
    \dd s \dd \sigma(y_{\Sigma});\\
  \label{def_C_eps}
    (C_\eps(\lm) f)(x_{\Sigma}, t) 
    & := 
    u(x_{\Sigma}, t) \int_{\dR^d} 
    G_\lm(x_{\Sigma} + \eps t \nu(x_{\Sigma}) - y) f(y) \dd y.
\end{align}
\end{subequations}
In order to investigate the properties of $A_{\varepsilon}(\lambda), B_{\varepsilon}(\lambda)$
and $C_\varepsilon(\lambda)$, we introduce several auxiliary operators.
For $\varepsilon \in (0, \beta]$ define the embedding operator
\begin{equation} \label{def_embedding}
  \mathcal{I}_\varepsilon:
  L^2(\Sigma \times (-\varepsilon, \varepsilon)) 
  \rightarrow L^2(\Omega_{\varepsilon}), \quad 
  (\mathcal{I}_\varepsilon\Phi)(x_{\Sigma} + t \nu(x_{\Sigma}))
  := \Phi(x_{\Sigma}, t).
\end{equation}
It follows from Hypothesis \ref{hypothesis_hypersurface} (b) and 
Proposition \ref{corollary_integration_tube} that the operator 
$\mathcal{I}_\varepsilon$ is bounded, everywhere defined
and bijective.
Its inverse is given by
\begin{equation*} 
  \mathcal{I}_\varepsilon^{-1}: L^2(\Omega_{\varepsilon}) \rightarrow 
  L^2(\Sigma \times (-\varepsilon, \varepsilon)), 
  \quad (\mathcal{I}_\varepsilon^{-1} h)(x_{\Sigma}, t) = 
  h(x_{\Sigma} + t \nu(x_{\Sigma})).
\end{equation*}
Furthermore, we consider the scaling operator 
\begin{equation} \label{def_S_eps}
  S_{\varepsilon}: L^2(\Sigma \times (-1, 1))
  \rightarrow L^2(\Sigma \times (-\varepsilon, \varepsilon)), \quad 
  (S_{\varepsilon} \Xi) (x_{\Sigma}, t) 
  := \frac{1}{\sqrt{\varepsilon}} \Xi \left( x_{\Sigma}, \frac{t}{\varepsilon} \right).
\end{equation}
The operator $S_{\varepsilon}$ is unitary and its inverse is 
\begin{equation*} 
  \begin{split}
    &S_{\varepsilon}^{-1}: L^2(\Sigma \times (-\varepsilon, \varepsilon)) 
    \rightarrow L^2(\Sigma \times (-1, 1)), \qquad (S_{\varepsilon}^{-1} \Phi)(x_{\Sigma}, t) = 
    \sqrt{\varepsilon} \Phi(x_{\Sigma}, \varepsilon t).
  \end{split}
\end{equation*}
In the next lemma some properties of $A_\varepsilon(\lambda), B_\varepsilon(\lambda)$
and $C_\varepsilon(\lambda)$ are provided.

\begin{lem} \label{lemma_ABC_eps}
  Let $\lambda \in (-\infty, 0)$ and let $\eps \in (0, \beta]$. 
  Moreover, let the operators $u_\varepsilon, v_\varepsilon$, $\cI_\varepsilon$
  and $S_{\varepsilon}$
  be defined by \eqref{def_u_eps}, \eqref{def_v_eps}, \eqref{def_embedding}
  and~\eqref{def_S_eps}, respectively, and let 
  the integral operators $A_\eps(\lm)$, $B_\eps(\lm)$ and $C_\eps(\lm)$
  be as in~\eqref{def_ABC_eps}.
  Then the following assertions are true.
  \begin{myenum}
    \item 
      It holds
      \begin{equation*}\label{eq:ACeps}
	      A_\eps(\lm) = 
    	  R(\lm) v_\eps \cI_\eps S_\varepsilon
	      \quad\text{and}\quad
    	  C_{\eps}(\lm) 
		  = 
    	  S_{\varepsilon}^{-1} \cI_\eps^{-1} 
	      u_{\varepsilon} R(\lm).
      \end{equation*}
      In particular, the operators $A_\eps(\lm)$ and $C_\eps(\lm)$
	  are bounded and everywhere defined.
    \item It holds
	\begin{equation*}
		B_\eps(\lm) = 
	    S_\eps^{-1} \cI_\eps^{-1} u_\eps 
    	R(\lm) v_\eps 
	    \cI_\eps S_\eps.
	\end{equation*}
	Moreover, for any $M \in (0, 1)$ there exists $\lambda_M < 0$
    such that $\| B_{\varepsilon}(\lambda) \| \leq M$ for all 
    $\lambda \in (-\infty, \lambda_M)$ and all $\eps \in (0, \beta]$. 
    In particular, $B_\eps(\lm)$ is bounded and everywhere 	
    defined  and for $\lm < \lm_M$    
	 the operator $1 - B_{\varepsilon}(\lambda)$ has a bounded
	 and everywhere defined inverse.
  \end{myenum}
\end{lem}

\begin{proof}
  Let $\varepsilon \in (0, \beta]$ and
  $\lambda \in (-\infty, 0)$ be fixed.

  (i) 
  We show the two formulae
  $A_\eps(\lm) = R(\lm) v_\eps \cI_\eps S_\varepsilon$
  and
  $C_{\eps}(\lm) = S_{\varepsilon}^{-1} \cI_\eps^{-1} u_{\varepsilon} R(\lm)$. 
  The operators $A_\eps(\lm)$ and $C_\eps(\lm)$ are then automatically bounded
  and everywhere defined, as the operators $S_\eps$, $\cI_\eps$
  as well as their inverses, $u_\eps$, $v_\eps$
  and $R(\lm)$
  have these properties.
  
  Let $\Xi \in L^2(\Sigma \times (-1, 1))$. 
  By the definitions of $\cI_\eps$ and
  $S_\eps$ it holds that
  \begin{equation*} 
    (\cI_\eps S_\eps \Xi)(y_{\Sigma} + s \nu(y_{\Sigma})) 
    = 
    \frac{1}{\sqrt{\eps}}
    \Xi\left(y_{\Sigma}, \frac{s}{\eps}\right).
  \end{equation*}
  Furthermore, the definitions of $v_\eps$ and $v$ (see \eqref{def_u_v})
  imply for any $h \in L^2(\Omega_\varepsilon)$ that
  \begin{equation*} 
    \begin{split}
      (v_\eps h)\left( y_{\Sigma} + s \nu(y_{\Sigma}) \right) 
      &= 
      \frac{1}{\sqrt{\eps}} 
      \,\sign V\left( y_{\Sigma} + \frac{\beta}{\eps} s \nu(y_{\Sigma}) \right) 
      \left|\beta V\left( y_{\Sigma} + \frac{\beta}{\eps}s\nu(y_{\Sigma})
      \right)\right|^{1/2} h\left( y_{\Sigma} + s \nu(y_{\Sigma}) \right) \\
      &= 
      \frac{1}{\sqrt{\eps}} v \left( y_{\Sigma}, \frac{s}{\eps} \right) h\left( y_{\Sigma} + s \nu(y_{\Sigma}) \right) 
    \end{split}
  \end{equation*}
  for a.e.
  $(y_{\Sigma}, s) \in \Sigma \times (-\eps, \eps)$. 
  Using the transformation 
  $\Omega_{\varepsilon} \ni y = y_{\Sigma} + s \nu(y_{\Sigma}) 
  \mapsto (y_{\Sigma}, s) \in \Sigma \times (-\eps, \eps)$
  and Proposition \ref{corollary_integration_tube} we find that
  \begin{equation} \label{label_lemma_transformation1}
    \begin{split}
      \big( R(\lm) v_{\varepsilon}
          \mathcal{I}_\varepsilon S_{\varepsilon} \Xi \big)(x)
          &= \int_{\mathbb{R}^d} G_{\lambda}(x - y) (v_{\varepsilon}
          \mathcal{I}_\varepsilon S_{\varepsilon} \Xi)(y) \mathrm{d} y \\
      &= \int_{\Sigma} \int_{-\varepsilon}^{\varepsilon} G_{\lambda}(x - y_{\Sigma} - s \nu(y_{\Sigma})) 
          \frac{1}{\varepsilon} (v \Xi)\left(y_{\Sigma}, \frac{s}{\varepsilon}\right) \det(1 - s W(y_{\Sigma})) 
          \mathrm{d} s \mathrm{d} \sigma(y_{\Sigma}) \\
      &= \int_{\Sigma} \int_{-1}^1 G_{\lambda}(x - y_{\Sigma} - \varepsilon r \nu(y_{\Sigma})) v\left( y_{\Sigma}, r \right) 
          \det(1 - \varepsilon r W(y_{\Sigma})) \Xi\left(y_{\Sigma}, r \right) \mathrm{d} r \mathrm{d} \sigma(y_{\Sigma}) \\
      &= (A_{\varepsilon}(\lambda) \Xi)(x).
    \end{split}
  \end{equation}
  Since this is true for a.e. $x \in \dR^d$, the 
  first formula of item (i) is shown.
	
  Next, we show the assertion on $C_\varepsilon(\lambda)$.
  A simple calculation yields that
  \begin{equation} \label{label_lemma_transformation3}
  \begin{split}
      (S_\eps^{-1} \cI_\eps^{-1} 
          u_\eps g)(x_{\Sigma}, t) 
      &
      = 
      \sqrt{\eps} (
      \cI_\eps^{-1} u_\eps g)(x_{\Sigma}, \eps t) \\
      &= 
      \sqrt{\eps} \cdot 
      \frac{1}{\sqrt{\eps}} 
      \left| \beta V\left( x_{\Sigma} +
      \frac{\beta}{\eps} \eps t \nu(x_{\Sigma}) \right) \right|^{1/2} 
      g(x_{\Sigma} + \eps t \nu(x_{\Sigma})) \\
      &= 
      u(x_{\Sigma}, t) g(x_{\Sigma} + \varepsilon t \nu(x_{\Sigma}))
  \end{split}
  \end{equation}
  for any $g \in L^2(\dR^d)$  and 
  a.e. $(x_{\Sigma}, t) \in \Sigma \times (-1, 1)$.
  Using \eqref{label_lemma_transformation3}, we find that
  \begin{equation*}
  \begin{split}
      (S_{\eps}^{-1} \cI_\varepsilon^{-1} 
      u_{\varepsilon} R(\lm) f)(x_{\Sigma}, t) 
      &= u(x_{\Sigma}, t) \int_{\dR^d} 
      G_{\lambda}(x_{\Sigma} + \varepsilon t \nu(x_{\Sigma}) - y) 
      f(y) \mathrm{d} y = (C_\varepsilon(\lambda) f)(x_\Sigma, t)
    \end{split}
  \end{equation*}
  for all $f \in L^2(\dR^d)$ 
  and a.e. $(x_{\Sigma}, t) \in \Sigma \times (-1, 1)$. 
  Thus, the second formula is shown as well.
			
  (ii) 
  Using \eqref{label_lemma_transformation3} and
  \eqref{label_lemma_transformation1} we get that
  \begin{equation*}
    \begin{split}
      \big( S_{\varepsilon}^{-1} \mathcal{I}_\varepsilon^{-1}
          u_{\varepsilon} R(\lm) v_{\varepsilon} \mathcal{I}_\varepsilon 
          S_{\varepsilon} \Xi\big)(x_{\Sigma}, t) 
      &= u(x_{\Sigma}, t) \big(R(\lambda) v_{\varepsilon} 
          \mathcal{I}_\varepsilon S_{\varepsilon} \Xi\big) (x_{\Sigma} + \varepsilon t \nu(x_{\Sigma})) \\
      &= u(x_{\Sigma}, t) \int_{\Sigma} \int_{-1}^1 
          G_{\lambda}(x_{\Sigma} + \varepsilon t \nu(x_{\Sigma}) - y_{\Sigma} - \varepsilon r \nu(y_{\Sigma})) 
          v\left( y_{\Sigma}, r \right) \\
      &\qquad \qquad \qquad \qquad \qquad \qquad \cdot 
          \det(1 - \varepsilon r W(y_{\Sigma})) \Xi\left(y_{\Sigma}, r \right) 
          \mathrm{d} r \mathrm{d} \sigma(y_{\Sigma}) \\
      &= (B_{\varepsilon}(\lambda) \Xi)(x_{\Sigma}, t).
    \end{split}
  \end{equation*}
  Therefore, we obtain the desired formula in item (ii).
  In particular, this formula implies
  that $B_{\varepsilon}(\lambda)$ is bounded 
  and everywhere defined. 
  
  Let $M \in (0, 1)$ be fixed. 
  Note that $\cI_\varepsilon$ 
  and 
  $\cI_\varepsilon^{-1}$ 
  are uniformly bounded for all 
  $\varepsilon \in (0, \beta]$ by 
  Proposition~\ref{corollary_integration_tube} and 
  Hypothesis~\ref{hypothesis_hypersurface} (b).
  Furthermore, recall that $S_\varepsilon$ is unitary.
  Hence, using Proposition \ref{prop_first_resolvent_formula} (ii) we find that
  \begin{equation*}
    \begin{split}
      \| B_\varepsilon(\lambda) \| &= 
      \big\| S_{\varepsilon}^{-1} \mathcal{I}_\varepsilon^{-1} 
      	  u_{\varepsilon}R(\lm) v_{\varepsilon}
         \mathcal{I}_\varepsilon S_{\varepsilon} \big\| 
       \leq \big\| S_{\varepsilon}^{-1} \big\|\cdot \big\| \mathcal{I}_\varepsilon^{-1} \big\| \cdot 
         \big\| u_{\varepsilon} R(\lm) v_{\varepsilon} \big\| 
         \cdot \| \mathcal{I}_\varepsilon \| \cdot \| S_{\varepsilon} \| \leq M
     \end{split}
  \end{equation*}
  for all $\lambda < 0$ with $\vert\lambda\vert$ sufficiently large 
  and all  $\varepsilon \in (0, \beta]$.
\end{proof}

After all these preparatory steps it is simple to transform 
the resolvent formula for $H_\varepsilon$ 
from Proposition~\ref{prop_first_resolvent_formula} 
into another one which is more convenient for the 
convergence analysis.

\begin{prop} \label{theorem_resolvent_formula}
  Let $H_\varepsilon$ be defined as in \eqref{def_H_eps}
  and let $A_{\varepsilon}(\lambda), B_{\varepsilon}(\lambda)$ and 
  $C_{\varepsilon}(\lambda)$ be as in \eqref{def_ABC_eps}.
   Then there exists a $\lambda_0 < 0$ such that 
  $(-\infty, \lambda_0) \subset \rho(H_\varepsilon)$ and
  \begin{equation*}
  	(H_\varepsilon - \lambda)^{-1} = 
  	R(\lm) + 
    A_{\varepsilon}(\lambda) \big(1 - B_{\varepsilon}(\lambda)\big)^{-1}
    C_{\varepsilon}(\lambda)
  \end{equation*}
  for all $\varepsilon \in (0, \beta]$ and $\lambda < \lambda_0$.
\end{prop}

\begin{proof}
  Let the operators $u_\varepsilon$, $v_\varepsilon$, $\cI_\varepsilon, S_\varepsilon$
  be defined as in \eqref{def_u_eps}, \eqref{def_v_eps},
  \eqref{def_embedding} and \eqref{def_S_eps}, respectively.  
  Choose $\lambda_0 < 0$ such that 
  $1 -u_{\varepsilon} R(\lambda) v_{\varepsilon}$
  and $1 - B_\varepsilon(\lambda)$ are boundedly invertible for
  any $\lambda < \lambda_0$ and all $\varepsilon \in (0, \beta]$
  (such a $\lambda_0$ exists by 
  Proposition~\ref{prop_first_resolvent_formula}
  and Lemma \ref{lemma_ABC_eps}).
  Then, it holds by Proposition \ref{prop_first_resolvent_formula}
  and Lemma \ref{lemma_ABC_eps} that
  \begin{equation*}
    \begin{split}
      (H_\varepsilon- \lambda)^{-1} 
          &= R(\lambda) + R(\lambda) v_{\varepsilon} 
          \left( 1 - u_{\varepsilon} R(\lambda) v_{\varepsilon} \right)^{-1} 
          u_{\varepsilon} R(\lambda) \\
      &= R(\lambda) + A_{\varepsilon}(\lambda) S_{\varepsilon}^{-1} \mathcal{I}_\varepsilon^{-1} 
          \left( 1 - \mathcal{I}_\varepsilon S_{\varepsilon} B_{\varepsilon}(\lambda) S_{\varepsilon}^{-1} 
          \mathcal{I}_\varepsilon^{-1} \right)^{-1} 
          \mathcal{I}_\varepsilon S_{\varepsilon} C_{\varepsilon}(\lambda) \\
      &= R(\lambda) + A_{\varepsilon}(\lambda) 
          \left( 1 - B_{\varepsilon}(\lambda) \right)^{-1} C_{\varepsilon}(\lambda),
    \end{split}
  \end{equation*}
  which proves the statement of this proposition.
\end{proof}

\subsection{Proof of Theorem \ref{theorem_convergence}} \label{subsection_convergence}

In this section we prove Theorem~\ref{theorem_convergence}. The argument essentially reduces to special the case $Q=0$, 
which will be treated first.
In this situation  we have to show that the family of operators $H_\varepsilon$ 
converges in the norm resolvent sense
to $A_{\delta, \alpha}$, as $\varepsilon \rightarrow 0+$.
Because of Theorem~\ref{theorem_delta_op} and Proposition \ref{theorem_resolvent_formula} it is sufficient to investigate
the convergence of $A_{\varepsilon}(\lambda), B_{\varepsilon}(\lambda)$ and 
$C_{\varepsilon}(\lambda)$ separately. This is done in the following lemma. 
In the proof we make use of the integral estimates in Appendix~\ref{appa} and we frequently use the Schur test for integral operators; cf. 
\cite[Example III 2.4]{kato} or \cite[Satz 6.9]{weidmann1}.

\begin{lem} \label{proposition_convergence}
  Let $\lambda \in (-\infty, 0)$ and
  let $A_{\varepsilon}(\lambda)$, $B_{\varepsilon}(\lambda)$ and $C_{\varepsilon}(\lambda)$ 
  be defined as in \eqref{def_ABC_eps}. 
  Then there exists a constant  $c = c(d, \lambda, \Sigma, V) > 0$ 
  such that for all sufficiently small $\eps > 0$ the following estimates hold:
  \[
  		\| A_\eps(\lm) - A_0(\lm) \|,\, \| C_\eps(\lm) - C_0(\lm) \|  
  		\leq c \eps \big( 1 + |\ln \varepsilon| \big)^{1/2},
  		\qquad
	    \| B_{\varepsilon}(\lambda) - B_0(\lambda) \| \leq 
   		 c \varepsilon \big( 1 + |\ln \varepsilon| \big).
  \]
\end{lem}

\begin{proof}
  First, we provide an estimate related to the Weingarten map $W(y_{\Sigma})$.
  Let $\mu_1(y_{\Sigma}), \dots, \mu_{d-1}(y_{\Sigma})$ be the eigenvalues of the
  matrix of $W(y_{\Sigma})$, which are independent of the 
  parametrization of $\Sigma$ and uniformly bounded on $\Sigma$;
  cf. Proposition \ref{proposition_eigenvalues_Weingarten_map}. 
  This implies for $s \in (-1, 1)$ and $\varepsilon \in (0, 1)$ 
  the existence of a constant $c_1 > 0$ such that
  \begin{equation} \label{label_eigenvalues_Weingarten}
      D_\eps(y_\Sigma,s) 
      := \left| 1 - \det(1 - \varepsilon s W(y_{\Sigma})) \right|
      = \left| 1 - \prod_{k=1}^{d-1} (1 - \varepsilon s \mu_k(y_{\Sigma})) \right| 
      \leq c_1 \varepsilon,\quad y_\Sigma \in\Sigma.
  \end{equation}
  Fix $\lambda \in (-\infty, 0)$. 
  In order to find an estimate for 
  $\| A_\eps(\lambda) - A_0(\lambda) \|$, we introduce for
  $\varepsilon > 0$ the auxiliary integral operator 
  $\wh{A}_{\eps}(\lambda)\colon L^2(\Sigma \times (-1, 1)) \rightarrow L^2(\dR^d)$ by
  \begin{equation*}
  	\big(\wh{A}_{\varepsilon}(\lambda) \Xi\big)(x) 
	:= 
    \int_{\Sigma} \int_{-1}^1 
    G_{\lambda}(x - y_{\Sigma} - \varepsilon s \nu(y_{\Sigma})) 
    v(y_{\Sigma}, s) \Xi(y_{\Sigma}, s) \dd s \dd \sigma(y_{\Sigma}).
  \end{equation*} 
  The quantities
  $\| A_{\varepsilon}(\lambda) - \wh{A}_{\varepsilon}(\lambda) \|$ 
  and 
  $\| \wh{A}_{\varepsilon}(\lambda) - A_0(\lambda)\|$
  are estimated separately for sufficiently small $\eps >0$.
  To estimate $\| \wh{A}_{\varepsilon}(\lambda) - A_{\varepsilon}(\lambda)\|$
  we introduce the functions $\wt F_{1,\eps}^A\colon \dR^d\arr [0, \infty]$ and
  $\wh F_{1,\eps}^A\colon \Sigma\times (-1,1)\arr [0, \infty]$ by
   \begin{align*}
      \wt F_{1,\eps}^A(x) & := \int_{\Sigma} \int_{-1}^1 \big|
          G_{\lambda}(x - y_{\Sigma} - \varepsilon s \nu(y_{\Sigma})) v(y_{\Sigma}, s) \big|
          D_\eps(y_\Sigma,s)
          \dd s \dd \sigma(y_{\Sigma}),\\
	  \wh F_{1,\eps}^A(y_{\Sigma},s) & := \int_{\dR^d} \big|
	  G_{\lambda}(x - y_{\Sigma} - \varepsilon s \nu(y_{\Sigma})) v(y_{\Sigma}, s) \big|
          D_\eps(y_\Sigma,s)\dd x.
	\end{align*}
  With the aid of \eqref{label_eigenvalues_Weingarten},
  Proposition~\ref{proposition3}\,(i) and Proposition \ref{proposition_appendix_hat_1} (ii),
  we find that
  \[
		\sup \wt F_{1, \eps}^A \le c_{A,1}\eps
		\qquad\text{and}\qquad
		\sup \wh F_{1, \eps}^A \le c_{A,1}\eps
  \]	
  with a constant $c_{A, 1} = c_{A,1}(d, \lm, \Sigma, V) > 0$.
  Using these bounds and the Schur test
  we obtain
  \begin{equation}\label{eq:estA1}
    \begin{split}
    \big\| \wh{A}_{\varepsilon}(\lambda) - A_{\varepsilon}(\lambda) \big\|^2  
	\le \big(\sup \wt F_{1, \eps}^A\big)\cdot\big( \sup \wh F_{1, \eps}^A\big)
	\le c_{A,1}^2\eps^2.
    \end{split}
  \end{equation}
  To estimate $\|\wh{A}_{\varepsilon}(\lambda) - A_0(\lambda)\|$,  
  we introduce the functions $\wt F_{2,\eps}^A\colon \dR^d\arr [0, \infty]$ and
  $\wh F_{2,\eps}^A\colon \Sigma\times (-1,1)\arr [0, \infty]$ by
  \begin{align*}
	\wt F_{2,\eps}^A (x) 
	& := 
       \int_{\Sigma} \int_{-1}^1 
        \big| \big(
        G_{\lambda}(x - y_{\Sigma} - \varepsilon s \nu(y_{\Sigma})) - 
        G_{\lambda}(x - y_{\Sigma})\big) v(y_{\Sigma}, s) \big| 
        \dd s \dd \sigma(y_{\Sigma}),\\
        \wh F_{2,\eps}^A (y_{\Sigma}, s) & := 
         \int_{\dR^d} \big| \big(
          G_{\lambda}(x - y_{\Sigma} - \varepsilon s \nu(y_{\Sigma})) - G_{\lambda}(x - y_{\Sigma})\big)
          v(y_{\Sigma}, s) \big|  \dd x.
  \end{align*}      
  We find with the help of 
  Proposition \ref{proposition_appendix_3} (i) and Proposition \ref{proposition_appendix_1} 
  that
  \[
		\sup \wt F_{2, \eps}^A \le c_{A,2}\eps( 1 + |\ln \eps|)
		\qquad\text{and}\qquad
		\sup \wh F_{2,\eps}^A \le c_{A,2}\eps
  \]	
  with a constant $c_{A, 2} = c_{A,2}(d, \lm, \Sigma, V) > 0$.
  Using these bounds and the Schur test we get
  \begin{equation}\label{eq:estA2}
      \big\| \wh{A}_{\varepsilon}(\lambda) - A_0(\lambda) \big\|^2 \le
      \big(\sup \wt F_{2,\eps}^A\big)\cdot \big( \sup \wh F_{2,\eps}^A\big)
      \le c_{A,2}^2\eps^2(1+ |\ln\eps|).
  \end{equation}
  Combining the estimates~\eqref{eq:estA1}, \eqref{eq:estA2}
  and applying the triangle inequality for the operator norm, we conclude that there exists a 
  constant $c_A = c_A(d, \lambda, \Sigma, V) > 0$ such that  
  \begin{equation*}
  \begin{split}
    \big\| A_{\varepsilon}(\lambda) - A_0(\lambda) \big\| 
        &\leq 
     \big\| A_{\varepsilon}(\lambda) - \wh{A}_{\varepsilon}(\lambda) \big\| 
     +
     \big\| \wh{A}_{\varepsilon}(\lambda) - A_0(\lambda) \big\| \\
     &\leq  c_{A,1}\eps + 
     c_{A,2}\eps(1+|\ln \eps|)^{1/2}\le
     c_A \varepsilon 
     \big( 1 + |\ln \eps| \big)^{1/2}.
     \end{split}
  \end{equation*}

  Next, we analyze the convergence of $B_{\varepsilon}(\lambda)$. 
  For this purpose, we introduce for $\varepsilon > 0$ the auxiliary operator 
  $\wh{B}_{\varepsilon}(\lambda)\colon L^2(\Sigma \times (-1, 1)) 
          \rightarrow L^2(\Sigma \times (-1, 1))$ as 
  \begin{equation*}
   \big(\wh{B}_{\varepsilon}(\lambda) \Xi\big)(x_{\Sigma}, t) 
          := u(x_{\Sigma}, t) \int_{\Sigma} \int_{-1}^1 
          G_{\lambda}(x_{\Sigma} + \varepsilon t \nu(x_{\Sigma}) - y_{\Sigma} - \varepsilon s \nu(y_{\Sigma}))
          v(y_{\Sigma}, s) \Xi(y_{\Sigma}, s) \mathrm{d} s \mathrm{d} \sigma(y_{\Sigma}).
   \end{equation*}
   As in the analysis of convergence of $A_\eps(\lm)$, we separately prove the estimates 
  for  $\| B_{\varepsilon}(\lambda) - \wh{B}_{\varepsilon}(\lambda)\|$
  and $\| \wh{B}_{\varepsilon}(\lambda) - B_0(\lambda)\|$, which yield then the claimed convergence result.
  To estimate  $\| B_{\varepsilon}(\lambda) - \wh{B}_{\varepsilon}(\lambda)\|$, we introduce the  functions
  $\wt F_{1,\eps}^B, \wh F_{1,\eps}^B\colon \Sigma\times (-1,1)\arr [0, \infty]$ by
  \begin{align*}
	\wt F_{1, \eps}^B (x_\Sigma,t) &:=  
	\int_\Sigma \int_{-1}^1 \big|
	u(x_\Sigma,t) G_\lm(x_\Sigma + \eps t \nu(x_\Sigma) - y_\Sigma -\eps s\nu(y_\Sigma))v(y_\Sigma,s) \big|
       D_\eps(y_\Sigma,s)\dd s \dd \s(y_\Sigma),\\
	\wh F_{1, \eps}^B (y_\Sigma,s) & := \int_\Sigma \int_{-1}^1 \big|
	u(x_\Sigma,t) G_\lm(x_\Sigma + \eps t \nu(x_\Sigma) - y_\Sigma -\eps s\nu(y_\Sigma))v(y_\Sigma,s) \big|
       D_\eps(y_\Sigma,s) \dd t \dd \s(x_\Sigma). 
  \end{align*}
  Using  \eqref{label_eigenvalues_Weingarten} and Proposition~\ref{proposition3}\,(ii),	we get 
  \[
  \sup \wt F_{1,\eps}^B \le c_{B,1}\eps\qquad\text{and}\qquad
		\sup \wh F_{1,\eps}^B \le c_{B,1}\eps
	\]
    with a constant $c_{B, 1} = c_{B,1}(d, \lm, \Sigma, V)  >0$.	
  Employing these bounds and again applying the Schur test we obtain that
  \begin{equation}\label{eq:estB1}
   	\big\| B_{\varepsilon}(\lambda) - \wh{B}_{\varepsilon}(\lambda) \big\|^2 
	\le \big (  \sup \wt F_{1,\eps}^B \big )\cdot
	\big (  \sup \wh F_{1,\eps}^B \big)\le c_{B,1}^2\eps^2.
  \end{equation}
  To estimate $\| \wh{B}_{\varepsilon}(\lambda) - B_0(\lambda) \|$ we introduce
  the auxiliary function 
  $F_{2,\eps}^B \colon \Sigma\times(-1,1)\arr [0, \infty]$
  as
  \[
	F_{2,\eps}^B(x_\Sigma,t) := \int_{\Sigma} \int_{-1}^1 \big| u(x_{\Sigma}, t) \big(
         G_{\lambda}(x_{\Sigma} + \varepsilon t \nu(x_{\Sigma}) - y_{\Sigma} - \varepsilon s \nu(y_{\Sigma})) - G_{\lambda}(x_{\Sigma} - y_{\Sigma})\big)v(y_\Sigma, y)\big|
       \dd  s \dd \s(y_\Sigma).                
  \]
  Using that the absolute value of the integral kernel of $\wh{B}_{\varepsilon}(\lambda) - B_0(\lambda)$
  is symmetric and applying  Proposition~\ref{proposition_appendix_3} (ii), we obtain
  with the help of the Schur test that
  \begin{equation}\label{eq:estB2}
  \big\| \wh{B}_{\varepsilon}(\lambda) - B_0(\lambda) \big\| \le \sup F_{2,\eps}^B
      \le c_{B, 2} \eps \big( 1 + |\ln \varepsilon| \big)
  \end{equation}
  with a constant $c_{B, 2} = c_{B,2}(d, \lm, \Sigma, V) >0$. 
  Putting together the estimates in \eqref{eq:estB1}, \eqref{eq:estB2} and employing the triangle inequality, we eventually 
  deduce that there is a constant $c_B = c_B(d, \lambda, \Sigma, V) > 0$ such that
  \begin{equation*}
    \begin{split}
      \| B_{\varepsilon}(\lambda) - B_0(\lambda) \|  &\leq \big\| B_{\varepsilon}(\lambda) - \wh{B}_{\varepsilon}(\lambda) \big\| 
          + \big\| \wh{B}_{\varepsilon}(\lambda) - B_0(\lambda) \big\| \\
         & \leq c_{B, 1} \varepsilon + c_{B, 2} \varepsilon \big( 1 + |\ln \varepsilon| \big) \leq 
      c_B \eps \big( 1 + |\ln \varepsilon| \big).
    \end{split}
  \end{equation*}

  Finally, we analyze the convergence of $C_{\varepsilon}(\lambda)$. 
  Using again the Schur test, Proposition \ref{proposition_appendix_3}~(i)
  and Proposition~\ref{proposition_appendix_1} we find
  \begin{equation*}
    \begin{split}
      \big\| C_{\varepsilon}(\lambda) - C_0(\lambda) \big\|^2 &\leq \sup_{y \in \mathbb{R}^d} 
          \int_{\Sigma} \int_{-1}^1 \left|u(x_{\Sigma}, t) \big(
          G_{\lambda}(x_{\Sigma} + \varepsilon t \nu(x_{\Sigma}) - y) - G_{\lambda}(x_{\Sigma} - y)\big) 
          \right| \mathrm{d} t \mathrm{d} \sigma(x_{\Sigma}) \\
      &\qquad\cdot \sup_{(x_{\Sigma}, t) \in \Sigma \times (-1, 1)} \int_{\mathbb{R}^d}  \left|u(x_{\Sigma}, t) 
          \big(G_{\lambda}(x_{\Sigma} + \varepsilon t \nu(x_{\Sigma}) - y) - G_{\lambda}(x_{\Sigma} - y)\big) 
          \right| \mathrm{d} y\\
      &
           \le c_C\eps^2\big( 1 + |\ln \varepsilon| \big)
    \end{split}
  \end{equation*}
  with a constant $c_C = c_C(d, \lambda, \Sigma, V) > 0$. Setting $c := \max\{c_A, c_B, c_C \}$, 
  the claimed result of this lemma follows.
\end{proof}

Theorem~\ref{theorem_delta_op}, Proposition~\ref{theorem_resolvent_formula} 
and Lemma~\ref{proposition_convergence} contain the essential 
ingredients to prove Theorem \ref{theorem_convergence}. In the first two steps of the proof the special case $Q=0$ is discussed, the general
situation is treated in the last step. 

\begin{proof}[Proof of Theorem \ref{theorem_convergence}]
  For $\lambda < 0$ and $\varepsilon \in [0, \beta]$ the operators 
  $A_{\varepsilon}(\lambda)$, $B_{\varepsilon}(\lambda)$ and 
  $C_{\varepsilon}(\lambda)$ are defined as in~\eqref{def_ABC_eps}.

  {\it Step 1:} First, we prove that $(1 - B_0(\lambda))^{-1}$
  exists and is bounded and everywhere defined for $\lm < 0$
  with $\vert\lambda\vert$ sufficiently large. Let $M \in (0, 1)$ be fixed and choose 
  $\lambda_M < 0$ such that $\| B_\varepsilon(\lambda) \| \leq M$ 
  for any $\lambda < \lambda_M$ and all $\varepsilon \in (0, \beta]$; 
  recall that such a $\lambda_M$
  exists by Lemma~\ref{lemma_ABC_eps}\,(ii). Hence,
  the operators $(1 - B_\varepsilon(\lambda))^{-1}$
  are bounded and everywhere defined for $\lambda < \lambda_M$ and $\varepsilon \in (0, \beta]$, 
  and it holds  
  \begin{equation} \label{equation_main_theorem1}
    \big\| (1 - B_\varepsilon(\lambda))^{-1} \big\| \leq \frac{1}{1 - M}.
  \end{equation}
  Because of~\eqref{equation_main_theorem1} and
  Lemma~\ref{proposition_convergence}, 
  we can apply \cite[Theorem IV 1.16]{kato},
  which yields that $1 - B_0(\lambda)$ is boundedly invertible. Moreover, 
  for $\lambda < \lambda_M$ we conclude from \cite[Theorem IV 1.16]{kato} and \eqref{equation_main_theorem1} that
  \begin{equation*}
  \begin{split}
      \big\| (1 - B_\varepsilon(\lambda))^{-1} - (1 - B_0(\lambda))^{-1} \big\|
        &\leq \frac{  \| (1 - B_\varepsilon(\lambda))^{-1} \|^2}
        {1 - \| B_\varepsilon(\lambda) - B_0(\lambda) \| \| (1 - B_\varepsilon(\lambda))^{-1} \|} \| B_\varepsilon(\lambda) - B_0(\lambda) \|\\
        &\leq c_1 \varepsilon \big( 1 + |\ln \varepsilon| \big)
  \end{split}
  \end{equation*}
  holds for all $\varepsilon > 0$ sufficiently small and a constant 
  $c_1 = c_1(d,\lm,\Sigma,V) > 0$.  
  
   {\it Step 2:} From
  Proposition~\ref{theorem_resolvent_formula},
  Lemma~\ref{proposition_convergence} and the estimates in Step 1 we obtain
  \begin{equation*}
    \begin{split}
      \big\| (H_\varepsilon - \lambda)^{-1} &- \big( R(\lm) + 
          A_0(\lambda) (1 - B_0(\lambda))^{-1} C_0(\lambda) \big) \big\| \\
       &= \big\| A_{\varepsilon}(\lambda) (1 - B_{\varepsilon}(\lambda))^{-1} C_{\varepsilon}(\lambda) - 
           A_0(\lambda) (1 - B_0(\lambda))^{-1} C_0(\lambda) \big\| 
           \leq c_2 \varepsilon \big( 1 + |\ln \varepsilon| \big)
    \end{split}
  \end{equation*}
  with a constant $c_2 = c_2(d,\lm,\Sigma,V) > 0$. 
   It remains to verify that
  \begin{equation}\label{okok}
     R(\lm) + A_0(\lambda) (1 - B_0(\lambda))^{-1} C_0(\lambda) 
        = (A_{\delta, \alpha} - \lambda)^{-1},
  \end{equation}
  where $\alpha\in L^\infty(\Sigma)$ is defined as in the theorem,
  \begin{equation*}
  \alpha(x_{\Sigma}) = \int_{-\beta}^\beta V(x_{\Sigma} + s \nu(x_{\Sigma})) \dd s.
  \end{equation*}
  In order to show \eqref{okok} let $u$ and $v$ be given by \eqref{def_u_v},
  introduce the bounded operators
  $\wh{U}\colon  L^2(\Sigma) \rightarrow L^2(\Sigma \times (-1, 1))$ and 
  $\wh{V}\colon  L^2(\Sigma \times (-1, 1)) \rightarrow L^2(\Sigma)$ by
  \begin{equation*}
    (\wh{U} \xi) (x_{\Sigma}, t) := u(x_\Sigma, t) \xi(x_{\Sigma})\quad\text{and}\quad
	\big(\wh{V} \Xi\big)(x_{\Sigma}) := 
    \int_{-1}^1 v(x_{\Sigma}, s) \Xi(x_{\Sigma}, s) \dd s
  \end{equation*}
  defined
  almost everywhere, and note that $\wh{V} \wh{U}$ is the multiplication
  operator with $\alpha$ in $L^2(\Sigma)$. 
  Furthermore, recall the definition of the bounded 
  operators $\gamma(\lambda)$, $M(\lambda)$ and the formula for 
  $\gamma(\lambda)^*$ from \eqref{def_gamma_field_Weyl_function};
  cf. Theorem \ref{theorem_delta_op}.
  Then, we observe that 
  \begin{equation*}
    \begin{split}
      (A_0(\lambda) \Xi)(x) &= \int_{\Sigma} \int_{-1}^1 G_{\lambda}(x - y_{\Sigma}) v(y_{\Sigma}, s) 
          \Xi(y_{\Sigma}, s) \mathrm{d} s \mathrm{d} \sigma(y_{\Sigma}) \\
      &= \int_{\Sigma} G_{\lambda}(x - y_{\Sigma}) \bigg( \int_{-1}^1 v(y_{\Sigma}, s) 
          \Xi(y_{\Sigma}, s)\mathrm{d} s \bigg) \mathrm{d} \sigma(y_{\Sigma}) 
          = \big(\gamma(\lambda) \wh{V} \Xi \big)(x)
    \end{split}
  \end{equation*}
  for any $\Xi \in L^2(\Sigma \times (-1, 1))$ 
  and a.e.  $x \in \dR^d$. Thus, we conclude 
  $A_0(\lambda) = \gamma(\lambda) \wh{V}$. 
  In a similar way, one finds that
  \begin{equation*}
    \begin{split}
      (B_0(\lambda) \Xi)(x_{\Sigma}, t) &= u(x_{\Sigma}, t) \int_{\Sigma} \int_{-1}^1 
          G_{\lambda}(x_{\Sigma} - y_{\Sigma}) v(y_{\Sigma}, s) \Xi(y_{\Sigma}, s) 
          \mathrm{d} s \mathrm{d} \sigma(y_{\Sigma}) \\
      &= u(x_{\Sigma}, t) \int_{\Sigma} G_{\lambda}(x_{\Sigma} - y_{\Sigma}) 
          \bigg( \int_{-1}^1 v(y_{\Sigma}, s) \Xi(y_{\Sigma}, s) \dd s \bigg) 
          \mathrm{d} \sigma(y_{\Sigma}) 
      = \big(\wh{U} M(\lambda) \wh{V} \Xi \big)(x_{\Sigma}, t)
    \end{split}
  \end{equation*}
  for any $\Xi \in L^2(\Sigma \times (-1, 1))$ 
  and a.e. $(x_\Sigma,t)\in \Sigma \times (-1, 1)$, 
  which implies $B_0(\lambda) = \wh{U} M(\lambda) \wh{V}$.
  Finally, one sees that
  \begin{equation*}
    \begin{split}
      (C_0(\lambda) f)(x_\Sigma,t) &= u(x_{\Sigma}, t) \int_{\mathbb{R}^d} 
          G_{\lambda}(x_{\Sigma} - y) f(y) \mathrm{d} y = 
          \bigl( \wh{U} \gamma(\lambda)^* \bigr) (x_{\Sigma}, t)
    \end{split}
  \end{equation*}
  for all $f \in L^2(\mathbb{R}^d)$ and a.e.
  $(x_{\Sigma}, t) \in \Sigma \times (-1, 1)$, which implies 
  $C_0(\lambda) = \wh{U} \gamma(\lambda)^*$. 
  Thus, we get
  \begin{equation}\label{usethis}
    \lim_{\eps\arr 0} (H_\varepsilon - \lambda)^{-1} = R(\lambda) +
       A_0(\lambda) \big(1 - B_0(\lambda)\big)^{-1} C_0(\lambda)  
       = R(\lambda) + \gamma(\lambda) \wh{V} \big(1 - \wh{U} M(\lambda) \wh{V}\big)^{-1} \wh{U} \gamma(\lambda)^*.
  \end{equation}
  Since $\alpha=\wh{V} \wh{U}$ we conclude from
  \begin{equation*}
    \wh{V} \Big(1 - \wh{U} M(\lambda) \wh{V}\big)^{-1}\wh U =
     \big(1 - \wh{V} \wh{U} M(\lambda) \big)^{-1} \wh{V}\wh U=\big(1 - \alpha M(\lambda) \big)^{-1} \alpha,
  \end{equation*}
  and \eqref{usethis} together with Theorem~\ref{theorem_delta_op} that
  \begin{equation*}
       \lim_{\varepsilon \rightarrow 0+} (H_\varepsilon - \lambda)^{-1} 
      = R(\lm) + \gamma(\lambda) \big( 1 - \alpha M(\lambda)\big)^{-1} \alpha \gamma(\lambda)^*  
	= 
	(A_{\delta, \alpha} - \lambda)^{-1}.
  \end{equation*}
  This completes the proof of Theorem~\ref{theorem_convergence} in the case $Q=0$.

  {\it Step 3:} Let $Q \in L^{\infty}(\mathbb{R}^d)$ be real valued
  and let $\lambda \in \mathbb{R}$ be such that 
  $\lambda < \lambda_M - \| Q \|_{L^\infty}$.
  Then $\lambda$ is smaller than the lower bound of
  the operators $A_{\delta, \alpha} + Q$ and $H_\varepsilon + Q$ 
  for all $\varepsilon \in (0, \beta]$. 
  Using the formula
  \begin{equation*}
    (H_\varepsilon + Q - \lambda)^{-1} = 
    \Big(1 -    (H_\varepsilon + Q - \lambda)^{-1}Q\Big)    
    (H_\eps  - \lambda)^{-1},
  \end{equation*}
  we compute
  \begin{equation*}
    \begin{split}
      (H_\varepsilon + Q - \lambda)^{-1} & - (A_{\delta, \alpha} + Q - \lambda)^{-1}
      \\ 
       & = \Big[
       (H_\varepsilon + Q - \lambda)^{-1}(A_{\delta,\alpha}  +Q -\lambda)
       -1\Big](A_{\delta,\alpha}  +Q -\lambda)^{-1}\\
     & =  
     \Big[(H_\varepsilon + Q - \lambda)^{-1} - 
     \big(1 -    (H_\varepsilon + Q - \lambda)^{-1}Q\big)
     (A_{\delta,\alpha} - \lm)^{-1}\Big]
     (A_{\delta,\alpha} - \lambda)(A_{\delta,\alpha} + Q - \lambda)^{-1}
     \\
     &=
     \big(1 - (H_\varepsilon + Q - \lambda)^{-1}Q\big)
     \big[
     (H_\varepsilon  - \lambda)^{-1} - (A_{\delta, \alpha}  - \lambda)^{-1}  \big]
     \bigl(1-Q(A_{\delta,\alpha}  +Q -\lambda)^{-1}\bigr).
    \end{split}
  \end{equation*}
  This implies
  \begin{equation*}
    \begin{split}
      \big\| (H_\varepsilon &+ Q - \lambda)^{-1} - (A_{\delta, \alpha} + Q - \lambda)^{-1} \big\| \\
      &\leq \big\|1 - (H_\varepsilon + Q - \lambda)^{-1}Q \big\| \cdot \big\| (H_\varepsilon - \lambda)^{-1}
         - (A_{\delta, \alpha} - \lambda)^{-1} \big\| \cdot \big\| 1-Q(A_{\delta,\alpha}  +Q -\lambda)^{-1} \big\|.
    \end{split}
  \end{equation*}
  Since the norm 
  $\big\| 1 - (H_\varepsilon + Q - \lambda)^{-1}Q \big\|$ is uniformly bounded 
  in $\varepsilon$, the result of Step~2 yields the desired claim.
  \end{proof}

\subsection{Consequences for the spectra of $H_\eps$ and $A_{\delta, \alpha}$} 
\label{subsection_consequences}

In this section we discuss how Theorem~\ref{theorem_convergence} can be used to
deduce certain spectral properties of $H_\eps$ 
from those of $A_{\delta, \alpha}$ and vice versa. 

First, the approximation result in Theorem~\ref{theorem_convergence}
combined with the results in~\cite{BEL14_JPA, EI01, EK03} on the existence
of geometrically induced bound states for Schr\"odinger operators
with $\delta$-interactions supported on curves and surfaces
can be used to show the existence of such bound states 
also for the operators with regular potentials 
(in the approximating sequence) provided that the potential
well is sufficiently ``narrow'' and ``deep''. This application
is motivated by Open Problem 7.1 in the review paper~\cite{E08}.
In order to formulate the respective claims we introduce several
geometric notions.
We say that the hypersurface $\Sigma\subset\dR^d$, $d \ge 2$, is a 
\emph{local deformation} of the hypersurface $\wt\Sigma\subset\dR^d$,
if $\Sigma \ne \wt\Sigma$ and if there exists a compact set $K\subset\dR^d$ such that
$\Sigma\setminus K = \wt\Sigma\setminus K$.
Furthermore, we introduce for $\theta \in ( 0, \pi/2)$ 
the \emph{broken line} $\cL\subset\dR^2$ and the \emph{infinite circular conical surface} $\cC\subset\dR^3$ by
\[
	\cL := \big\{(x,y) \in\dR^2\colon y = \cot(\theta) |x|\big\},
	\qquad
	\cC := 
	\big\{(x,y,z) \in\dR^3\colon 
	z = \cot(\theta) \sqrt{x^2 + y^2}\big\}. 
\]

\begin{prop}\label{proposition_bound_states}
	Let $\Sigma\subset\dR^d$ be 
	as in Definition~\ref{definition_hypersurface} such that
	Hypothesis~\ref{hypothesis_hypersurface} holds and 
	assume that $\Sigma$ satisfies one of the following assumptions. 
	\begin{myenum}
		\item [\rm (a)] 
		$d = 2$ and $\Sigma$ is a local deformation of a straight line;
		\item [\rm (b)] 
		$d = 2$ and $\Sigma$ is a local deformation of a broken line;
		\item [\rm (c)] 
		$d = 3$ and $\Sigma$ is a $C^4$-smooth
		local deformation of a plane, which admits a global natural
		parametrization in the sense of \cite[Section 2-3, Definition 2]{dC};
		\item [\rm (d)] 
		$d = 3$ and $\Sigma$ is a local deformation of 
		an infinite circular conical surface.
	\end{myenum}
	Let the layer $\Omega_\beta\subset\dR^d$ be as in~\eqref{def_tube}
	and let $V \in L^\infty(\dR^d)$ be real valued 
	with $\supp V \subset \Omega_\beta$. 
	Assume that $\alpha \in L^\infty(\Sigma)$ 
	associated to $V$ as in Theorem~\ref{theorem_convergence}
	is a positive constant, which in case {\rm (c)} is assumed to be sufficiently large. 
	Then for all sufficiently small $\eps > 0$ 
	the self-adjoint operator $H_\eps$ in~\eqref{def_H_eps} 
	has a non-empty discrete spectrum below the threshold of its essential 
	spectrum. 
\end{prop}

\begin{proof}
	As there is no danger of confusion, we denote the value
	of the constant positive function $\aa\in L^\infty(\Sigma)$
	by $\aa$ as well.
	First, recall that by Theorem~\ref{theorem_convergence}
	the self-adjoint operators $H_\eps$
	converge in the norm resolvent sense 
	to the self-adjoint lower-semibounded operator 	$A_{\delta,\alpha}$
	as $\eps\arr 0+$. 	
	
	Next, in all the cases it is known that 
	$\sess(A_{\delta, \alpha}) = [-\alpha^2/4, \infty)$;
	\cf~\cite[Theorems 4.2 and 4.10]{BEL14_RMP} and~\cite[Theorem 3.3]{BEL14_JPA}. 	
	Moreover, Proposition~\ref{prop_first_resolvent_formula}
	implies that the operators $H_\eps$ 
	are bounded from below uniformly in $\eps \in (0,\beta]$.
	Hence, the result \cite[Satz 9.24\,(a)]{weidmann1} yields
	\begin{equation}\label{eq:ess}
		\inf\sess(H_\eps) \arr -\aa^2/4, \qquad\eps \arr 0+.
	\end{equation}	
	Moreover, in all the cases it is known that
	$\sd(A_{\delta, \alpha})\cap (-\infty,-\alpha^2/4) \neq \emptyset$;
	\cf \cite[Theorem 5.2]{EI01} for {\rm (a), (b)},
	\cite[Theorem 4.3]{EK03} for {\rm (c)} and \cite[Theorem 3.3]{BEL14_JPA} 
	for {\rm (d)}. Choose now a finite interval $(a,b)\subset (-\infty,-\aa^2/4)$
	with $a,b\in\rho(A_{\delta,\alpha})$ such that 
	$\sd(A_{\delta, \alpha})\cap (a,b) \ne \emptyset$.
	By~\eqref{eq:ess} we get for all sufficiently small $\eps > 0$ that
	$\inf\sess(H_\eps) > b$.

	Eventually, it follows from \cite[Satz 2.58\,(a) and Satz 9.24\,(b)]{weidmann1} 
	that for all sufficiently small $\eps > 0$ 
	the dimensions of the spectral subspaces of $H_\eps$ corresponding to $(a,b)$
	are equal to the dimension of the spectral subspace of $A_{\delta,\alpha}$ 			
	corresponding to the same interval; i.e.
	\begin{equation}\label{eq:subspaces}
		\dim \ran E_{H_\eps}((a,b)) = \dim\ran E_{A_{\delta,\alpha}}((a,b)) > 0.
	\end{equation} 
	Since $b < \inf\sess(H_\eps)$, this implies the claimed result.
\end{proof}
In the next proposition we show that Schr\"odinger operators with potential wells supported in curved periodic 
strips have gaps in their spectra under reasonable assumptions.
This proposition can be proven in a similar way as Proposition~\ref{proposition_bound_states}.
It suffices to combine a result~\cite[Corollary 2.2]{exner_yoshitomi_2001} on the existence of gaps 
in the negative spectrum for the Schr\"odinger operator with a strong $\delta$-interaction 
supported on a periodic curve with the main result of this paper and with standard
statements on spectral convergence.
Note that a similar idea was earlier used in a different,
albeit related context in \cite{Yo98}.
\begin{prop}
	Let $\Sigma \subset \dR^2$ be as in Definition~\ref{definition_hypersurface} 
	such that Hypothesis~\ref{hypothesis_hypersurface} holds. 
	Suppose that $\Sigma$ is parametrized via the tuple $\{\varphi,\dR,\dR^2\}$
	with $|\varphi'| \equiv 1$.
	For $\varphi(s)  = (\varphi_1(s),\varphi_2(s))$ define the signed curvature of 
	$\Sigma$ by $\kappa := \varphi_1'' \varphi_2' - \varphi_1' \varphi_2''$.
	Assume that $\Sigma$ is not a straight line and that
	$\kappa$ satisfies the following conditions. 
	\begin{myenum}	
		\item [\rm (a)] $\kappa \in C^2(\dR)$;
		\item [\rm (b)] $\kappa(s + L) = \kappa(s)$ for some $L > 0$ and all $s\in\dR$;
		\item [\rm (c)] $\int_0^L\kappa(s) \dd s = 0$;
		\item [\rm (d)] $\big|\int_0^T\kappa(s) \dd s\big| < \pi/2$ for all $T \in [0,L]$.
	\end{myenum}
	Let $\Omega_\beta\subset\dR^2$ be as in~\eqref{def_tube}
	and let $V \in L^\infty(\dR^d)$ be real valued with $\supp V \subset \Omega_\beta$.
	Assume that $\alpha \in L^\infty(\Sigma)$ associated to $V$ 
	as in Theorem \ref{theorem_convergence} is a sufficiently large positive constant. 		
	Then for all sufficiently small~$\eps > 0$ 
	the self-adjoint operator $H_\eps$ in~\eqref{def_H_eps}	has 
	at least one spectral gap in the interval $(-\infty, 0)$.
\end{prop}
One can also use Theorem~\ref{theorem_convergence}
to obtain spectral results for $A_{\delta, \alpha}$ from those of $H_\eps$.
To illustrate this idea, we show in the following example that 
for $d \geq 3$ the operator $A_{\delta, \alpha}$ does not have any negative bound states if
$\Sigma$ is a sphere with radius $R > 0$ 
and if $\alpha \in (0, \frac{d - 2}{R})$ is a constant.
The same result is also contained in~\cite[Theorem 2.3]{antoine_gesztesy_shabani_1987} 
for $d = 3$ and in~\cite[Theorem 4.1]{albeverio_kostenko_malamud_neidhardt_2013} for arbitrary 
$d \geq 3$. The proofs there are of a different nature than ours.
\begin{example}
	\rm Let $d \geq 3$, $R > 0$ and $\Sigma = \partial B(0, R)$.
	Let $\aa \in (0, \frac{d - 2}{R})$ be fixed.
	Let $q \in C^\infty_0(\dR_+)$ be non-negative 
	with $\supp q \subset (R/2, 3R/2)$ such that
	\begin{equation*}
		\int_{R/2}^{3R/2} q(r) \dd r = \aa.
	\end{equation*}
	%
	%
	Define the radially symmetric potential $V \in L^\infty(\dR^d)$ by 
	$V(x) := q(|x|)$, $x\in\dR^d$.
	Furthermore, define the scaled potentials $V_\eps$ and
	the operators $H_\eps$ 
	as in the Introduction with $\beta = R/2$
	and $\eps\in(0, R/2]$.
	By Theorem~\ref{theorem_convergence} the operators $H_\eps$ converge
	in the norm resolvent sense to $A_{\delta, \alpha}$ as $\eps \arr 0+$.
	For any $\eps \in (0, R/2]$ 
	the potential $V_\eps$ is also radially symmetric and 
	we get with the help of Lebesgue's dominated convergence
	theorem that for $q_\eps(|x|) := V_\eps(x)$
	\begin{equation*}
	\begin{split}
		\lim_{\eps \arr 0+}\int_0^\infty r q_\eps(r) \dd r 
		&= 
		\lim_{\eps \arr 0+}
		\int_{-\eps}^{\eps}
		\left(R + r\right) \frac{R}{2\eps} q\left(R + \frac{Rr}{2\eps} \right) \dd r\\
		& = 
		\lim_{\eps \arr 0+}
		\int_{R/2}^{3R/2} \left(\frac{2\eps z}{R} + R  - 2\eps \right) q(z) \dd z
		= 
		\aa R < d -2.
	\end{split}
	\end{equation*}
	Hence, using Bargmann's estimate~\cite[Theorem 3.2]{seto_1974}, 
	we obtain that 
	$H_\eps$ has no negative eigenvalues for all sufficiently
	small $\eps > 0$.
	Because of Theorem~\ref{theorem_convergence} 
	and \cite[Satz 9.24\,(a)]{weidmann1}
	it follows that $A_{\delta, \alpha}$ has no negative eigenvalues 
	as well. 
\end{example}

\begin{appendix}
\section{Estimates related to Green's function}\label{appa}

In this appendix we provide estimates for integrals that contain 
Green's function 
\begin{equation} \label{def_G_lambda2}
  G_{\lambda}(x) 
  = \frac{1}{(2 \pi)^{d/2}} \left( \frac{|x|}{\sqrt{-\lambda}} \right)^{1 - d/2} 
  K_{d/2 - 1} \left( \sqrt{-\lambda} |x| \right), 
  \quad x \in \mathbb{R}^d \setminus \{ 0 \},\quad d\geq 2,
\end{equation}
from \eqref{def_G_lambda}.
The estimates are formulated in a way such that they can be applied directly
in the main part of the paper. We note that some of the estimates below are known, but exact references are difficult to find in the mathematical literature.
Therefore, in order to keep this paper self-contained we also provide complete proofs of standard estimates as, e.g., in Proposition~\ref{proposition_asymptotics_G_lambda}.   
Throughout this appendix we assume that $\Sigma$ is an orientable $C^2$-hypersurface 
which satisfies Hypothesis \ref{hypothesis_hypersurface}, 
and we denote by $\nu$ the continuous unit 
normal vector field of $\Sigma$.

In the first preliminary proposition we discuss the asymptotics of $G_\lambda$
and $\nabla G_\lambda$.

\begin{prop} \label{proposition_asymptotics_G_lambda}
  Let $\lambda \in (-\infty, 0)$ and let $G_\lambda$
  be as in \eqref{def_G_lambda2}. Then there exists a constant $R = R(d) > 0$
  such that the following assertions hold.
  \begin{itemize}
    \item[\rm (i)] There is a constant $c = c(d)$ such that
    for all $|x| \leq \frac{R}{\sqrt{-\lambda}}$
    \begin{equation*}
      \big| G_\lambda(x) \big| \leq \begin{cases} c \big( 1 + \big| \ln \big( \sqrt{-\lambda} |x| \big) \big| \big), &\text{ if } d = 2, \\
                                                  c |x|^{2-d}, &\text{ if } d \geq 3, \end{cases}
    \end{equation*}
    holds. Moreover, there exists a constant $C = C(d)$ such that for all $|x| > \frac{R}{\sqrt{-\lambda}}$ it holds
    \begin{equation*}
      \big| G_\lambda(x) \big| \leq C \left( \frac{|x|}{\sqrt{-\lambda}} \right)^{1 - d/2} e^{-\sqrt{-\lambda} |x|}.
    \end{equation*}
    In particular, $G_\lambda \in L^1(\mathbb{R}^d)$.
    
    \item[\rm (ii)] The function $\mathbb{R}^d \setminus \{ 0 \} \ni x \mapsto G_\lambda(x)$
    is continuously differentiable. Furthermore, there exist constants $c = c(d, \lambda)$ and $C = C(d, \lambda)$ such that   
    for all $|x| \leq \frac{R}{\sqrt{-\lambda}}$
    \begin{equation*}
      \big| \nabla G_\lambda(x) \big| \leq c |x|^{1-d}
    \end{equation*}
    is true and for all $|x| > \frac{R}{\sqrt{-\lambda}}$ it holds
    \begin{equation*}
      \big| \nabla G_\lambda(x) \big| \leq C e^{-\sqrt{-\lambda} |x|}.
    \end{equation*}
    In particular, $ \nabla G_\lambda \in L^1(\mathbb{R}^d; \mathbb{C}^d)$.
   \end{itemize}
\end{prop}
\begin{proof}
  (i) Due to the asymptotic behavior of $K_{d/2-1}$, see \cite[Section 9.6 and 9.7]{abramowitz_stegun},
  there exist an $R = R(d) > 0$ and constants $C_1, C_2, C_3 > 0$ such that for any $p \in (0, R]$
  \begin{equation} \label{asymptotics_Bessel1}
    |K_{d/2-1}(p)| \leq \begin{cases} C_1 (1 + |\ln p|), & \text{ if }d=2, \\ 
    C_2 p^{1-d/2}, &\text{ if }d \geq 3, \end{cases}
  \end{equation}
  is satisfied and for any $p > R$
  \begin{equation} \label{asymptotics_Bessel2}
     |K_{d/2-1}(p)| \leq C_3 e^{-p}
  \end{equation}
  holds. Hence, the claimed asymptotics follow from the definition of $G_\lambda$. It is not difficult to check that 
  these asymptotics imply $G_\lambda\in L^1(\dR^d)$.

  (ii) Recall first that the mapping 
  $\mathbb{C} \setminus (-\infty, 0] \ni z \mapsto K_{\nu}(z)$ is holomorphic by \cite[Section 9.6]{abramowitz_stegun}.
  Therefore, $\mathbb{R}^d \setminus \{ 0 \} \ni x \mapsto G_\lambda(x)$ is continuously differentiable and
  we obtain
  \begin{equation} \label{asymptotics_Bessel3}
    \begin{split}
      \nabla G_\lambda(x) = \frac{1}{(2 \pi)^{d/2}} \frac{x}{|x|} \bigg(& \left(1-\frac{d}{2} \right) \frac{|x|^{-d/2}}{\sqrt{-\lambda}^{1-d/2}} 
          K_{d/2 - 1} \left( \sqrt{-\lambda} |x| \right) \\
      &+ \sqrt{-\lambda} \left( \frac{|x|}{\sqrt{-\lambda}} \right)^{1 - d/2} 
          K_{d/2 - 1}' \left( \sqrt{-\lambda} |x| \right) \bigg) 
    \end{split}
  \end{equation}
  for $d\geq 3$ and
  \begin{equation} \label{asymptotics_Bessel32}
      \nabla G_\lambda(x) =  \frac{1}{2 \pi} \frac{x}{|x|} \sqrt{-\lambda} \,
          K_{0}' \left( \sqrt{-\lambda} |x| \right)  
  \end{equation}
  in the case $d=2$.
  Since 
  \begin{equation} \label{asymptotics_Bessel4}
    K'_{d/2-1}(z) = K_{d/2}(z) + \frac{d-2}{2 z} K_{d/2-1}(z)
  \end{equation}
  by \cite[eq. 9.6.26]{abramowitz_stegun} we conclude from \eqref{asymptotics_Bessel1} that for $d\geq 3$ and $|x| \leq \frac{R}{\sqrt{-\lambda}}$ 
  \begin{equation*}
    \begin{split}
      \big| K'_{d/2-1}\big( \sqrt{-\lambda} |x| \big) \big| 
        &\leq \big| K_{d/2}\big( \sqrt{-\lambda} |x| \big) \big| 
        + \frac{d-2}{2 \sqrt{-\lambda} |x|} \big| K_{d/2-1}\big( \sqrt{-\lambda} |x| \big) \big| \\
      &\leq C_4 |x|^{-d/2} + C_5 \frac{d-2}{2 \sqrt{-\lambda} |x|} |x|^{1-d/2}
        \leq C_6 |x|^{-d/2}
      \end{split}
  \end{equation*}
  holds with some constant $C_6 = C_6(d, \lambda) > 0$. 
  It is easy to see that the same estimate is also true in the case $d=2$. 
  Hence, \eqref{asymptotics_Bessel1}, \eqref{asymptotics_Bessel3} and \eqref{asymptotics_Bessel32} yield
  \begin{equation*} 
    \begin{split}
      \left|\nabla G_\lambda(x)\right| &\leq C_7 \left(  |x|^{-d/2} \left|K_{d/2 - 1} \left( \sqrt{-\lambda} |x| \right) \right|
          + |x|^{1 - d/2} \left|K_{d/2 - 1}' \left( \sqrt{-\lambda} |x| \right) \right| \right)
          \leq c |x|^{1-d}.
    \end{split}
  \end{equation*}
  In the same way, using \eqref{asymptotics_Bessel2}, \eqref{asymptotics_Bessel3}, \eqref{asymptotics_Bessel32}, and \eqref{asymptotics_Bessel4} 
  one finds for all $|x| \geq \frac{R}{\sqrt{-\lambda}}$ that
  \begin{equation*}
    \big| \nabla G_\lambda(x) \big| \leq C e^{-\sqrt{-\lambda} |x|}
  \end{equation*}
  holds with some constant $C = C(d, \lambda)> 0$. Finally, it is not difficult to check that the asymptotics for $\nabla G_\lambda$ imply 
  $\nabla G_\lambda \in L^1(\mathbb{R}^d; \mathbb{C}^d)$.
\end{proof}

We start now with the estimates which are needed to show 
Lemma~\ref{proposition_convergence}. The first proposition 
in this context is only based on the integrability of $G_\lambda$; cf. Proposition~\ref{proposition_asymptotics_G_lambda}~(i). 

\begin{prop} \label{proposition_appendix_hat_1}
  Let $\Sigma$ be a $C^2$-hypersurface which satisfies 
  Hypothesis \ref{hypothesis_hypersurface} and let $\lambda \in (-\infty, 0)$.
  Then the following statements are true.
  \begin{itemize}
  \item[\rm (i)] There exists 
    a constant $C = C(d, \lambda) > 0$ such that
    \begin{equation*}
      \sup_{y_\Sigma \in \Sigma} \int_{\mathbb{R}^d} \left| G_{\lambda}(x - y_\Sigma) \right| 
      \mathrm{d} x \leq C.
    \end{equation*}
    \item[\rm (ii)] Let $\psi \in L^\infty(\Sigma \times (-1, 1))$
    and let $\varepsilon \in (0, \beta]$.
    Then there exists a constant $C = C(d, \lambda, \psi) > 0$ such that
    \begin{equation*}
      \sup_{(y_{\Sigma}, s) \in \Sigma \times (-1, 1)} \int_{\mathbb{R}^d} \big| 
          G_{\lambda}(x - y_{\Sigma} - \varepsilon s \nu(y_{\Sigma})) 
          \psi(y_{\Sigma}, s) \big| \mathrm{d} x \leq C.
    \end{equation*}
  \end{itemize}
\end{prop}

\begin{proof}
  We only prove item (ii), assertion (i) can be shown in the same way.
  For $(y_\Sigma, s) \in \Sigma \times (-1, 1)$ fixed we have
  \begin{equation*}
    \begin{split}
      \int_{\mathbb{R}^d} &\big| 
        G_{\lambda}(x - y_{\Sigma} - \varepsilon s \nu(y_{\Sigma})) 
        \psi(y_{\Sigma}, s) \big| \mathrm{d} x 
        \leq \| \psi \|_{L^\infty} \int_{\mathbb{R}^d} \big| 
        G_{\lambda}(x) \big| \mathrm{d} x \leq C
      \end{split}
  \end{equation*}
  with some constant $C = C(d, \lambda, \psi)$.  
  This is the claimed result. 
\end{proof}

The following lemma contains an auxiliary estimate associated to the hypersurface $\Sigma$.
Recall that $\sigma$ denotes the Hausdorff measure and $\Lambda_d$ is the $d$-dimensional Lebesgue measure.

\begin{lem} \label{lemma1}
  Let $\Sigma$ be a $C^2$-hypersurface which satisfies 
  Hypothesis \ref{hypothesis_hypersurface}. Then the following assertions are true.
  \begin{itemize}
    \item[\rm(i)] \it There exists a constant $\widetilde{C}_1 = \widetilde{C}_1(d, \Sigma) > 0$ 
    such that for any $r_0 > 0$  and any $x \in \mathbb{R}^d$ it holds
    \begin{equation*}
      \sigma(\Sigma \cap B(x, r_0)) \leq \widetilde{C}_1 r_0^{d-1}.
    \end{equation*} 
    
    \item[\rm(ii)] \it Let $\varepsilon \in (0, \beta]$ and let $\Omega_{\varepsilon}$ be as in 
    \eqref{def_tube}. 
    Then there exists a constant $\widetilde{C}_2 = \widetilde{C}_2(d, \Sigma) > 0$
    such that for any $r_0 > 0$  and any $x \in \mathbb{R}^d$ it holds
    \begin{equation*}
      \Lambda_d(\Omega_{\varepsilon} \cap B(x, r_0)) 
      \leq \widetilde{C}_2 \varepsilon r_0^{d-1}.
    \end{equation*} 
  \end{itemize}
\end{lem}
\begin{proof}
  (i) By the definition of the measure $\sigma$ we have
  \begin{equation*}
    \sigma(\Sigma \cap B(x, r_0)) = \int_{\Sigma \cap B(x, r_0)} \mathrm{d} \sigma(x_\Sigma) 
    = \sum_{i \in I} \int_{\varphi_i^{-1}(\Sigma \cap B(x, r_0))} 
      \chi_i(\varphi_i(u)) \sqrt{\det G_i(u)} \mathrm{d} u;
  \end{equation*}
  cf. \eqref{integral}. Assume that $\varphi_i^{-1}(\Sigma \cap B(x, r_0)) \neq \emptyset$, 
  let $u_i \in \varphi_i^{-1}(\Sigma \cap B(x, r_0))$ be fixed
  and let $u \in \varphi_i^{-1}(\Sigma \cap B(x, r_0))$ be arbitrary. 
  Using Hypothesis \ref{hypothesis_hypersurface} (c) we obtain
  \begin{equation*}
    c |u - u_i| \leq |\varphi_i(u) - \varphi_i(u_i)| 
    \leq |\varphi_i(u) - x| + |x - \varphi_i(u_i)| \leq 2 r_0.
  \end{equation*}
  Hence, it follows $\varphi_i^{-1}(\Sigma \cap B(x, r_0)) \subset B(u_i, 2 r_0 / c)$,
  which implies
  \begin{equation*}
    \begin{split}
      \sigma(\Sigma \cap B(x, r_0)) &= \sum_{i \in I} \int_{\varphi_i^{-1}(\Sigma \cap B(x, r_0))} 
          \chi_i(\varphi_i(u)) \sqrt{\det G_i(u)} \mathrm{d} u \\
      &\leq C_1 \sum_{i \in I} \int_{\varphi_i^{-1}(\Sigma \cap B(x, r_0))} \mathrm{d} u 
          \leq C_1 \sum_{i \in I} \int_{B(u_i, 2 r_0/c)} \mathrm{d} u \leq \widetilde{C}_1 r_0^{d-1},
    \end{split}
  \end{equation*}
  where we have used that $\det G_i$ is uniformly bounded and that the index set $I$ is finite by assumption.
  
  (ii) Using Proposition \ref{corollary_integration_tube} (ii)
  and Hypothesis \ref{hypothesis_hypersurface} (b) we find
  \begin{equation*}
    \begin{split}
      \Lambda_d(\Omega_{\varepsilon} \cap B(x, r_0)) 
          &= \int_{\Omega_{\varepsilon} \cap B(x, r_0)} \mathrm{d} y 
          = \int_{\Omega_{\varepsilon}} \mathds{1}_{\Omega_{\varepsilon} \cap B(x, r_0)}(y) \mathrm{d} y \\
      &= \int_{\Sigma} \int_{-\varepsilon}^{\varepsilon} 
          \mathds{1}_{\Omega_{\varepsilon} \cap B(x, r_0)}(y_{\Sigma} + s \nu(y_{\Sigma})) 
          \det(1 - s W(y_{\Sigma})) \mathrm{d} s \mathrm{d} \sigma(y_\Sigma) \\
      &\leq C_2 \int_{\Sigma} \int_{-\varepsilon}^{\varepsilon} 
          \mathds{1}_{\Omega_{\varepsilon} \cap B(x, r_0)}(y_{\Sigma} + s \nu(y_{\Sigma})) 
          \mathrm{d} s \mathrm{d} \sigma(y_\Sigma).
    \end{split}
  \end{equation*}
  Let $\iota_{\varphi_i}$ be given by \eqref{iotavarphi}. 
  Using the definition of the measure $\sigma$ and the fact that $\det G_i$ is 
  uniformly bounded by assumption, it follows
  \begin{equation*}
    \begin{split}
      \Lambda_d(\Omega_{\varepsilon} \cap B(x, r_0)) 
          &\leq C_2 \int_{\Sigma} \int_{-\varepsilon}^{\varepsilon} 
          \mathds{1}_{\Omega_{\varepsilon} \cap B(x, r_0)}(y_{\Sigma} + s \nu(y_{\Sigma})) 
          \mathrm{d} s \mathrm{d} \sigma(y_\Sigma) \\
      &= C_2 \sum_{i \in I} \int_{U_i} \int_{-\varepsilon}^{\varepsilon} \chi_i(\varphi_i(u)) 
          \mathds{1}_{\Omega_{\varepsilon} \cap B(x, r_0)}(\iota_{\varphi_i}(u, s)) 
          \sqrt{\det G_i(u)} \mathrm{d} s \mathrm{d} u \\
      &\leq C_3 \sum_{i \in I} \int_{U_i} \int_{-\varepsilon}^{\varepsilon} 
          \mathds{1}_{A_i}(u, s) \mathrm{d} s \mathrm{d} u,
    \end{split}
  \end{equation*}
  where $A_i := \iota_{\varphi_i}^{-1}(\Omega_{\varepsilon} \cap B(x, r_0)) 
  \subset U_i \times (-\varepsilon, \varepsilon)$ and hence 
  $\mathds{1}_{A_i} = \mathds{1}_{\Omega_{\varepsilon} \cap B(x, r_0)} \circ \iota_{\varphi_i}$.
  Assume $A_i \neq \emptyset$, let $(u_i, t) \in A_i$ be fixed and let $(u, s) \in A_i$ be arbitrary. 
  Using Hypothesis \ref{hypothesis_hypersurface} (c) we find
  \begin{equation*}
    \begin{split}
      c |u - u_i| &\leq c \left( |u - u_i|^2 + (s - t)^2 \right)^{1/2} 
          \leq \left| \varphi_i(u) + s \nu(\varphi_i(u)) - 
          \varphi_i(u_i) - t \nu(\varphi_i(u_i)) \right| \\
      &\leq \left| \varphi_i(u) + s \nu(\varphi_i(u)) - x \right| + \left| x - 
          \varphi_i(u_i) - t \nu(\varphi_i(u_i)) \right|  < 2 r_0.
    \end{split}
  \end{equation*}
  Therefore, we find $A_i \subset B(u_i, 2 r_0 / c) \times (-\varepsilon, \varepsilon)$, 
  which yields
  \begin{equation*}
    \begin{split}
      \Lambda_d(\Omega_{\varepsilon} \cap B(x, r_0)) 
          &\leq C_3 \sum_{i \in I} \int_{U_i} \int_{-\varepsilon}^{\varepsilon} 
          \mathds{1}_{A_i}(u, s) \mathrm{d} s \mathrm{d} u \\
      &\leq C_3 \sum_{i \in I} \int_{B(u_i, 2 r_0 / c)} \int_{-\varepsilon}^{\varepsilon}
      \mathrm{d} s \mathrm{d} u = \widetilde{C}_2 \varepsilon r_0^{d-1}.
    \end{split}
  \end{equation*}
\end{proof}

The next proposition contains two estimates of a similar flavour as in Proposition~\ref{proposition_appendix_hat_1}.
The proof is essentially based on Lemma~\ref{lemma1}.

\begin{prop} \label{proposition2}
  Let $\Sigma$ be a $C^2$-hypersurface which satisfies 
  Hypothesis \ref{hypothesis_hypersurface} and let $\lambda \in (-\infty, 0)$. Then the following assertions are true.
  \begin{itemize}
    \item[\rm (i)] There is a constant
    $C = C(d, \lambda, \Sigma) > 0$ such that
    \begin{equation*}
      \sup_{x \in \mathbb{R}^d} \int_{\Sigma} \big| G_{\lambda}(x - y_{\Sigma}) \big| 
          \mathrm{d} \sigma(y_{\Sigma}) \leq C.
    \end{equation*}
    Moreover, for any fixed $C > 0$ there exists $\lambda_C < 0$ such that 
    the above inequality is satisfied for this fixed $C$ and all $\lambda < \lambda_C$.
    \item[\rm (ii)] Let $\varepsilon \in (0, \beta]$ and let $\Omega_\varepsilon$ be as in \eqref{def_tube}.
    Then there is a constant $M = M(d, \lambda, \Sigma) > 0$ such that
    \begin{equation*}
      \sup_{x \in \mathbb{R}^d} \int_{\Omega_{\varepsilon}} |G_{\lambda}(x - y)| \mathrm{d} y  
      \leq M \varepsilon.
    \end{equation*} 
    Moreover, for any fixed $M > 0$ there exists $\lambda_M < 0$ such that 
    the above inequality is satisfied for this fixed $M$, all $\lambda < \lambda_M$ and all $\varepsilon \in (0, \beta]$.
  \end{itemize}
\end{prop}

\begin{proof}
  (i) Let $x \in \mathbb{R}^d$ and $\lambda \in (-\infty, 0)$ be fixed,
  let $R > 0$ be as in Proposition \ref{proposition_asymptotics_G_lambda}
  and set 
  \begin{equation*}
      \Sigma_1 = \bigl\{ y_\Sigma \in \Sigma: 
          \sqrt{-\lambda} |x - y_\Sigma| < R \bigr\}\quad\text{and}\quad
      \Sigma_2 = \bigl\{ y_\Sigma \in \Sigma: 
          \sqrt{-\lambda} |x - y_\Sigma| > R \bigr\}.
  \end{equation*}
  Then we have
  \begin{equation*}
    \int_{\Sigma} |G_{\lambda}(x - y_{\Sigma})| 
    \mathrm{d} \sigma(y_{\Sigma}) 
    = \int_{\Sigma_1} |G_{\lambda}(x - y_{\Sigma})| 
    \mathrm{d} \sigma(y_{\Sigma}) + \int_{\Sigma_2} |G_{\lambda}(x - y_{\Sigma})| 
    \mathrm{d} \sigma(y_{\Sigma}).
  \end{equation*}
  
  In order to find an estimate for the integral over $\Sigma_1$ observe that
  \begin{equation*}
    \int_{\Sigma_1} |G_{\lambda}(x - y_\Sigma)| \mathrm{d} \sigma(y_\Sigma)
    = \sum_{n=1}^{\infty} \int_{A_n} |G_{\lambda}(x-y_\Sigma)| \mathrm{d} \sigma(y_\Sigma),
  \end{equation*}
  where the sets $A_n$ are defined as
  \begin{equation*}
    A_n = \left\{ y_{\Sigma} \in \Sigma: 2^{-n} < \sqrt{-\lambda} 
        |x - y_{\Sigma}|/R < 2^{-n+1} \right\}, \qquad n \in \mathbb{N}.
  \end{equation*}
  Due to the asymptotics of $G_{\lambda}$ in Proposition \ref{proposition_asymptotics_G_lambda} (i) we find for $y_\Sigma \in A_n$ that
  \begin{equation*}
    |G_{\lambda}(x-y_\Sigma)| \leq \begin{cases} C_1 \big(1 + n \ln 2 \big), &\text{ if } d=2, \\ 
    C_2 (-\lambda)^{d/2-1} \cdot \big(2^{d-2}\big)^{n} , 
    &\text{ if } d \geq 3.\end{cases}
  \end{equation*}
  Hence, we get in the case $d \geq 3$ that
  \begin{equation*} 
    \begin{split}
      \int_{\Sigma_1} |G_{\lambda}(x - y_\Sigma)| \mathrm{d} \sigma(y_\Sigma)
      &= \sum_{n=1}^{\infty} \int_{A_n} |G_{\lambda}(x-y_\Sigma)| \mathrm{d} \sigma(y_\Sigma)
      \leq \sum_{n=1}^{\infty} C_2 (-\lambda)^{d/2-1} \cdot \big(2^{d-2}\big)^{n} \int_{A_n} \mathrm{d} \sigma(y_\Sigma).
    \end{split}
  \end{equation*}
  Using that $A_n \subset B\big(x, R \cdot 2^{-n+1}/\sqrt{-\lambda} \big)$, we can 
  employ Lemma \ref{lemma1} (i) and get
  \begin{equation} \label{proposition2_label3}
    \begin{split}
      \int_{\Sigma_1} |G_{\lambda}(x - y_\Sigma)| \mathrm{d} \sigma(y_\Sigma) 
      & \leq C_3 (-\lambda)^{d/2-1} 
          \sum_{n=1}^{\infty} \big(2^{d-2}\big)^{n} \cdot 
          \big(R \cdot 2^{-n+1}/\sqrt{-\lambda}\big)^{d-1} \\
      & \leq C_4 (-\lambda)^{-1/2}
          \sum_{n=1}^{\infty} 2^{-n} = C_4 (-\lambda)^{-1/2}.
    \end{split}
  \end{equation}

  Note that similar estimates are also true in the case $d=2$, as for $y_\Sigma \in A_n$ it holds
  \begin{equation*}
    |G_{\lambda}(x - y_\Sigma)| \leq C_1 \big(1 + n \ln 2 \big) \leq C_5 \big(1 + 2^{n/2} \big) 
    \leq 2 C_5 2^{n/2}.
  \end{equation*}
  
  Finally, we derive an estimate for the integral over $\Sigma_2$. 
  For this purpose we decompose the integral as follows
  \begin{equation*}
    \int_{\Sigma_2} |G_{\lambda}(x - y_\Sigma)| \mathrm{d} \sigma(y_\Sigma) 
    = \sum_{n=1}^{\infty} \int_{B_n} |G_{\lambda}(x - y_\Sigma)| \mathrm{d} \sigma(y_\Sigma),
  \end{equation*}
  where the sets $B_n$ are given by
  \begin{equation*}
     B_n = \left\{ y_{\Sigma} \in \Sigma: 
         R + n - 1 < \sqrt{-\lambda} |x - y_{\Sigma}| < R + n \right\}, \qquad n \in \mathbb{N}.
  \end{equation*}
  Using the asymptotics of $G_{\lambda}$ for large arguments
  from Proposition \ref{proposition_asymptotics_G_lambda} (i)
  we find for $y_\Sigma \in B_n$ that
  \begin{equation*}
    \big| G_\lambda(x - y_\Sigma) \big| 
        \leq C_6 \left( \frac{|x - y_\Sigma|}{\sqrt{-\lambda}} \right)^{1 - d/2} e^{-\sqrt{-\lambda} |x - y_\Sigma|}
        \leq C_7 (-\lambda)^{d/2-1} e^{-(R + n - 1)}.
  \end{equation*}
  Since $B_n \subset B\big( x, (R+n)/\sqrt{-\lambda} \big)$,
  we can employ Lemma \ref{lemma1} (i) and get that
  \begin{equation} \label{proposition2_label4}
    \begin{split}
      \int_{\Sigma_2} |G_{\lambda}(x - y_\Sigma)| \mathrm{d} \sigma(y_\Sigma) 
          &= \sum_{n=1}^{\infty} \int_{B_n} |G_{\lambda}(x-y_\Sigma)| \mathrm{d} \sigma(y_\Sigma) 
          \leq \sum_{n=1}^{\infty} C_7 (-\lambda)^{d/2-1} e^{-(R + n - 1)} \int_{B_n} \mathrm{d} \sigma(y_\Sigma) \\
      &\leq C_8 (-\lambda)^{-1/2} \sum_{n=1}^{\infty} e^{-n} (R+n)^{d-1} 
      = C_9 (-\lambda)^{-1/2}.
    \end{split}
  \end{equation}
  The estimates \eqref{proposition2_label3} and \eqref{proposition2_label4} 
  yield now the claimed bounds. In particular, by choosing $\lambda <0$ with $\vert\lambda\vert$ sufficiently large
  the constant $C := (C_4 + C_9) (-\lambda)^{-1/2}$ becomes arbitrarily small. 
  
  (ii) The proof of this statement is similar to the one of assertion (i). 
  The difference is to replace $\Sigma$ by the layer $\Omega_\varepsilon$ and to use
  Lemma \ref{lemma1} (ii) instead of Lemma \ref{lemma1} (i).
\end{proof}

The next result is a consequence of Proposition \ref{proposition2}.

\begin{prop} \label{proposition3}
  Let $\Sigma$ be a $C^2$-hypersurface which satisfies 
  Hypothesis \ref{hypothesis_hypersurface},
  let $\lambda \in (-\infty, 0)$ and let $\varepsilon \in (0, \beta]$. 
  Then the following statements are true.
  \begin{itemize}
    \item[\rm (i)] Let $\psi \in L^{\infty}(\Sigma \times (-1, 1))$.
    Then there exists 
    a constant $C= C(d, \lambda, \Sigma, \psi) > 0$ such that
    \begin{equation*}
      \sup_{x \in \mathbb{R}^d} \int_{\Sigma} \int_{-1}^1 
          \big| G_{\lambda}(x - y_{\Sigma} - \varepsilon s \nu(y_{\Sigma})) 
          \psi(y_\Sigma, s)\big| \mathrm{d} s \mathrm{d} \sigma(y_{\Sigma}) 
          \leq C.
    \end{equation*}
    \item[\rm (ii)] Let $\omega, \psi \in L^{\infty}(\Sigma \times (-1, 1))$.
    Then there exists a constant $C = C(d, \lambda, \Sigma, \omega, \psi) > 0$ such that
    \begin{equation*}
    \begin{split}
       \sup_{(x_\Sigma, t) \in \Sigma \times (-1, 1)} &\int_{\Sigma} \int_{-1}^1 \big| \omega(x_{\Sigma}, t) 
           G_{\lambda}(x_{\Sigma} + \varepsilon t \nu(x_{\Sigma}) - y_{\Sigma} - \varepsilon s \nu(y_{\Sigma})) 
           \psi(y_{\Sigma}, s)\big| \mathrm{d} s \mathrm{d} \sigma(y_{\Sigma}) 
       \leq C.
    \end{split}
  \end{equation*}
  \end{itemize}
\end{prop}

\begin{proof}
  Let $z \in \mathbb{R}^d$ be fixed. According to Hypothesis \ref{hypothesis_hypersurface} (b)
  there exists a constant $C_1>0$ such that 
  for any $s \in (-1, 1)$ and all $y_{\Sigma} \in \Sigma$ the estimate
  \begin{equation*}
    1 \leq C_1 \det(1 - \varepsilon s W(y_{\Sigma}))
  \end{equation*}
  holds. 
  Using 
  Proposition \ref{corollary_integration_tube} (ii) and 
  Proposition \ref{proposition2} (ii) we conclude
  that there exists a constant $C > 0$ (independent of $z$) such that
  \begin{equation} \label{proposition3_label1}
    \begin{split}
        \int_{\Sigma} \int_{-1}^1 \big| 
            G_{\lambda}(z - y_{\Sigma} - \varepsilon s \nu(y_{\Sigma}))
            \big| \mathrm{d} s \mathrm{d} \sigma(y_{\Sigma}) 
        \leq C_1 \int_{\Sigma} \int_{-1}^1 \big| 
            G_{\lambda}(z - y_{\Sigma} - \varepsilon s \nu(y_{\Sigma})) \big|
            \det(1 - \varepsilon s W(y_{\Sigma})) \mathrm{d} s \mathrm{d} \sigma(y_{\Sigma})& \\
        = \frac{C_1}{\varepsilon} \int_{\Sigma} \int_{-\varepsilon}^{\varepsilon}
            \big| G_{\lambda}(z - y_{\Sigma} - r \nu(y_{\Sigma})) \big| 
            \det(1 - r W(y_{\Sigma})) \mathrm{d} r \mathrm{d} \sigma(y_{\Sigma}) 
        = \frac{C_1}{\varepsilon} \int_{\Omega_{\varepsilon}} 
            \big| G_{\lambda}(z - y) \big| \mathrm{d} y \leq C&.
    \end{split}
  \end{equation}

  (i) Let $x \in \mathbb{R}^d$ be fixed. Then it follows from \eqref{proposition3_label1}
  (with $z = x$) that
  \begin{equation*} 
    \begin{split}
      \int_{\Sigma} \int_{-1}^1 &
          \big| G_{\lambda}(x - y_{\Sigma} - \varepsilon s \nu(y_{\Sigma})) 
          \psi(y_\Sigma, s)\big| \mathrm{d} s \mathrm{d} \sigma(y_{\Sigma}) 
      \leq \| \psi \|_{L^\infty} \int_{\Sigma} \int_{-1}^1 
          \big| G_{\lambda}(x - y_{\Sigma} - \varepsilon s \nu(y_{\Sigma})) 
          \big| \mathrm{d} s \mathrm{d} \sigma(y_{\Sigma}) \leq C,
    \end{split}
  \end{equation*}
  where $C$ does not depend on $x$. This is the claimed result.

  (ii) Let $x_\Sigma \in \Sigma$ and $t \in (-1, 1)$ be fixed. Then it follows from \eqref{proposition3_label1}
  (with $z = x_{\Sigma} + \varepsilon t \nu(x_{\Sigma})$) that
  \begin{equation*} 
    \begin{split}
      \int_{\Sigma} \int_{-1}^1 \big| \omega(x_{\Sigma}, t) &
           G_{\lambda}(x_{\Sigma} + \varepsilon t \nu(x_{\Sigma}) - y_{\Sigma} - \varepsilon s \nu(y_{\Sigma})) 
           \psi(y_{\Sigma}, s)\big| \mathrm{d} s \mathrm{d} \sigma(y_{\Sigma}) \\
      &\leq \| \omega \|_{L^\infty} \| \psi \|_{L^\infty} \int_{\Sigma} \int_{-1}^1 
          \big| G_{\lambda}(x_{\Sigma} + \varepsilon t \nu(x_{\Sigma}) - y_{\Sigma} - \varepsilon s \nu(y_{\Sigma})) 
          \big| \mathrm{d} s \mathrm{d} \sigma(y_{\Sigma}) \leq C,
    \end{split}
  \end{equation*}
  where $C$ does not depend on $x_\Sigma$ and $t$. This completes the proof of Proposition~\ref{proposition3}.
\end{proof}

The following estimates are slightly more involved than the previous ones, 
as from now on also the gradient of $G_{\lambda}$ has to be considered.
First, we provide a useful simple argument that is needed for the next results.
Assume that $\tau \in L^\infty(\mathbb{R}^d \times \Sigma \times (-1, 1)^2; \mathbb{R}^d)$ and that
$(x, y_\Sigma, s, t) \in \mathbb{R}^d \times \Sigma \times (-1, 1)^2$ is such that
\begin{equation*}
   x - y_\Sigma + \varepsilon \theta \tau(x, y_\Sigma, s, t) \neq 0 
\end{equation*}
for all $\theta \in [0, 1]$ and all $\varepsilon \in (0, \beta]$. 
Since $G_\lambda(y)$ is differentiable for $y \neq 0$ (see Proposition \ref{proposition_asymptotics_G_lambda} (ii))
it follows that the mapping
$[0, 1] \ni \theta \mapsto G_{\lambda}(x - y_{\Sigma} + \varepsilon \theta \tau(x, y_{\Sigma}, s, t))$
is differentiable and one has
\begin{equation*}
  \begin{split}
    G_{\lambda}(x - y_{\Sigma} + \varepsilon \tau(x, y_\Sigma, s, t)) - G_{\lambda}(x - y_{\Sigma}) 
        &= \int_0^1 \frac{\mathrm{d}}{\mathrm{d} \theta} 
        G_{\lambda}(x - y_{\Sigma} + \varepsilon \theta \tau(x, y_{\Sigma}, s, t)) \mathrm{d} \theta \\
    &= \int_0^1 \big\langle \nabla G_{\lambda}(x - y_{\Sigma} + \varepsilon \theta \tau(x, y_{\Sigma}, s, t)),
        \varepsilon \tau(x, y_{\Sigma}, s, t) \big\rangle \mathrm{d} \theta.
  \end{split}
\end{equation*}
Using the Cauchy-Schwarz inequality for vectors in $\mathbb{R}^d$ this leads to
\begin{equation} \label{estimate_diff_G_lambda}
  \begin{split}
    \big| G_{\lambda}(x- y_{\Sigma} + \varepsilon \tau(x, y_{\Sigma}, s, t)) - 
        G_{\lambda}(x - y_{\Sigma}) \big| 
    &\leq \varepsilon \| \tau \|_{L^{\infty}} 
        \int_0^1 \left| \nabla G_{\lambda}(x - y_{\Sigma} + \varepsilon \theta \tau(x, y_{\Sigma}, s, t)) \right| \mathrm{d} \theta.
  \end{split}
\end{equation}

\begin{prop} \label{proposition_appendix_1}
  Let $\Sigma$ be a $C^2$-hypersurface which satisfies 
  Hypothesis \ref{hypothesis_hypersurface},
  let $\psi \in L^{\infty}(\Sigma \times (-1, 1))$ 
  and let $\lambda \in (-\infty, 0)$. Then there exists 
  a constant $C = C(d, \lambda, \psi) > 0$ such that for all $\varepsilon \in (0, \beta]$ it holds
  \begin{equation*}
    \sup_{(y_{\Sigma}, s) \in \Sigma \times (-1, 1)} \int_{\mathbb{R}^d} 
        \left| \big(G_{\lambda}(x - y_{\Sigma} - \varepsilon s \nu(y_{\Sigma})) 
        - G_{\lambda}(x - y_{\Sigma})\big) \psi(y_{\Sigma}, s) \right| \mathrm{d} x 
        \leq C \varepsilon.
  \end{equation*}
\end{prop}
\begin{proof}
  Let $y_{\Sigma} \in \Sigma$ and $s \in (-1, 1)$ be fixed. Since 
  $x - y_{\Sigma} - \varepsilon s \theta \nu(y_{\Sigma}) \neq 0$ for a.e. $x \in \mathbb{R}^d$,
  it follows from~\eqref{estimate_diff_G_lambda} (with $\tau(x, y_\Sigma, s, t) = -s \nu(y_\Sigma)$ implying
  $\|\tau\|_{L^\infty} = 1$)
  \begin{equation*}
    \begin{split}
      \int_{\mathbb{R}^d} \big| \big(G_{\lambda}(x - y_{\Sigma} - \varepsilon s \nu(y_{\Sigma})) - 
          G_{\lambda}(x - y_{\Sigma})\big) \psi(y_{\Sigma}, s) \big| \mathrm{d} x 
      &\leq \varepsilon \| \psi \|_{L^{\infty}} \int_{\mathbb{R}^d} \int_0^1 \big| \nabla G_{\lambda}(x - y_{\Sigma} - \varepsilon s \theta \nu(y_{\Sigma})) \big|
          \mathrm{d} \theta \mathrm{d} x.
    \end{split}
  \end{equation*}
  Now consider the bijective transformation 
  $T: \mathbb{R}^d \times (0, 1) \rightarrow \mathbb{R}^d \times (0, 1)$ 
  given by
  \begin{equation*}
    \begin{pmatrix} \xi \\ \phi \end{pmatrix} = T (x, \theta) 
    := \begin{pmatrix} x - y_{\Sigma} - \varepsilon s \theta \nu(y_{\Sigma})
    \\ \theta \end{pmatrix}.
  \end{equation*}  
  Note that $T$ is differentiable and that its Jacobian matrix is given by
  \begin{equation*}
    D T (x,\theta) = \begin{pmatrix}  \mathds{1}_d & - \varepsilon s \nu(y_{\Sigma}) \\ 
    0 & 1 \end{pmatrix},
  \end{equation*}
  where $\mathds{1}_d$ is the identity matrix in $\mathbb{R}^{d \times d}$. 
  Hence $|\det D T(x,\theta)| = 1$ and we conclude that
  \begin{equation*}
    \begin{split}
      \int_{\mathbb{R}^d}\! \big| \big(G_{\lambda}(x - y_{\Sigma} - \varepsilon s \nu(y_{\Sigma})) \!-\! 
          G_{\lambda}(x - y_{\Sigma})\big) \psi(y_{\Sigma}, s) \big| \mathrm{d} x 
      \!\leq\! \varepsilon \| \psi \|_{L^{\infty}}\! \int_{\mathbb{R}^d}\! \int_0^1\! \big| \nabla G_{\lambda}(x - y_{\Sigma} - \varepsilon s \theta \nu(y_{\Sigma})) \big|
          \mathrm{d} \theta \mathrm{d} x& \\
      = \varepsilon \| \psi \|_{L^{\infty}} \int_{\mathbb{R}^d} \int_0^1 \big| \nabla G_{\lambda}(\xi) \big|
          \mathrm{d} \phi \mathrm{d} \xi 
          = \varepsilon \| \psi \|_{L^{\infty}} \int_{\mathbb{R}^d} \big| \nabla G_{\lambda}(\xi) \big| \mathrm{d} \xi&,
    \end{split}
  \end{equation*}
  where we used in the last step that the integrand was independent of $\phi$. 
  Since the last integral is finite by Proposition~\ref{proposition_asymptotics_G_lambda}~(ii), the proof is complete.
\end{proof}

The next lemma contains an auxiliary estimate that is needed to 
prove the final two integral bounds.

\begin{lem} \label{lemma2}
  Let $\Sigma$ be a $C^2$-smooth hypersurface satisfying 
  Hypothesis \ref{hypothesis_hypersurface}, let $\lambda \in (-\infty, 0)$, let $r_0 > 0$ and let $\widetilde{\Sigma} \subset \Sigma$
  be measurable.
  Then there exists a constant $C = C(r_0, d, \lambda, \Sigma) > 0$ such that 
  \begin{equation*}
    \int_{\widetilde{\Sigma}} \int_{-1}^1 \
        \big| \nabla G_\lambda(x - y_\Sigma - \varepsilon \theta s \nu(y_\Sigma)) \big| \mathrm{d} s \mathrm{d} \sigma(y_\Sigma)
       \leq C
  \end{equation*}
  for all sufficiently small $\varepsilon > 0$, all $\theta \in (0, 1)$ and all 
  $x \in \mathbb{R}^d$ satisfying $\mathrm{dist}(x, \widetilde{\Sigma}) \geq r_0$.
\end{lem}
\begin{proof}
	Let $x \in \dR^d$ with $\mathrm{dist}(x, \widetilde{\Sigma}) \geq r_0$
	and $\theta \in (0, 1)$ be fixed. 
	We are going to show first that
	\begin{equation} \label{prop_app_3_label1}
	\begin{split}
		\int_{\widetilde{\Sigma}}
		\big| \nabla G_\lm\big(x - y_\Sigma - \eps\tt s \nu(y_\Sigma)) \big|
        \dd \s(y_\Sigma)  
        \leq \widetilde{C}
	\end{split}
	\end{equation}
     with $\wt C$ independent of $x, s, \theta$ and $\varepsilon$. 
  For this define the sets
  \begin{equation*}
     A_n := \left\{ y_{\Sigma} \in \widetilde{\Sigma}: 
         r_0 + n - 1 < |x - y_{\Sigma}| < r_0 + n \right\}, \qquad n \in \mathbb{N}.
  \end{equation*}
   Since it holds for $\eps \le  \frac{r_0}{2}$
    and all $(y_\Sigma,s) \in \widetilde{\Sigma}\times (-1,1)$ 
	\begin{equation*} 
		\big| x - y_\Sigma - \eps\tt s \nu(y_\Sigma) \big|
	    \geq 
	    | x - y_\Sigma|  - \eps \ge\frac{r_0}{2},
	\end{equation*}
	%
 	it follows from Proposition~\ref{proposition_asymptotics_G_lambda}\,(ii)
	that for $y_\Sigma \in A_n \subset \widetilde{\Sigma}$
	\begin{equation*}\label{bound_green_grad}
	   	\big| 
	   	\nabla G_\lm(x - y_\Sigma - \eps\tt s \nu(y_\Sigma)) 
    	\big|
        \leq C_1 e^{-\sqrt{-\lambda} |x - y_{\Sigma} - \eps\tt s \nu(y_\S)|}
	    \leq C_2 e^{-\sqrt{-\lambda} |x - y_\Sigma|} 
		\leq C_2 e^{-\sqrt{-\lambda} (r_0 + n - 1)}.
    \end{equation*}
  Finally, since $A_n \subset B( x, r_0+n)$
  we can employ Lemma \ref{lemma1} (i) and get that
  \begin{equation*}
    \begin{split}
      \int_{\widetilde{\Sigma}}
		\big| \nabla G_\lm\big(x - y_\Sigma - \eps\tt s \nu(y_\Sigma)) \big|
        &\dd \s(y_\Sigma)  
      = \sum_{n=1}^\infty \int_{A_n} \big| \nabla G_\lm\big(x - y_\Sigma - \eps\tt s \nu(y_\Sigma)) \big| \mathrm{d} \sigma(y_\Sigma) \\
          &\leq C_2 \sum_{n=1}^{\infty} \int_{A_n} e^{-\sqrt{-\lambda} (r_0 + n - 1)} \mathrm{d} \sigma(y_\Sigma) 
      \leq C_3 \sum_{n=1}^{\infty} e^{-\sqrt{-\lambda} \cdot n} (r_0+n)^{d-1} 
      = C_4.
    \end{split}
  \end{equation*}
	Thus, the estimate \eqref{prop_app_3_label1} is true. Finally, interchanging the order of integration we obtain
	\begin{equation*} 
    \begin{split}
	   \int_{\widetilde{\Sigma}} \int_{-1}^1 
        \big| \nabla G_\lambda(x - y_\Sigma - \varepsilon \theta s \nu(y_\Sigma)) \big| \mathrm{d} s \mathrm{d} \sigma(y_\Sigma)
        \leq \int_{-1}^1
        C_4 \mathrm{d} s \leq C.
	\end{split}
	\end{equation*}
     This is the claimed result.
\end{proof}

The next lemma contains the main tool to
prove the final two integral bounds.
Here, we apply Lemma~\ref{lemma2} for $\widetilde{\Sigma} = \Sigma$
and $\widetilde{\Sigma} = \Sigma \setminus B(x, r_0)$ with $x \in \mathbb{R}^d$.
Moreover, we make use of the notation $C([0, 1]; \mathbb{R}^d)$ for the set of all continuous
vector valued functions defined on the interval $[0, 1]$.

\begin{lem} \label{proposition_difficult}
  Let $\Sigma$ be a $C^2$-smooth hypersurface satisfying 
  Hypothesis \ref{hypothesis_hypersurface} and let $\lambda \in (-\infty, 0)$.
  Then there exists a constant $C = C(d, \lambda, \Sigma) > 0$ such that for all $\varepsilon > 0$ sufficiently small
  and all $x \in C([0, 1]; \mathbb{R}^d)$ it holds
  \begin{equation*}
     \int_{\Sigma} \int_{-1}^1 \int_0^1 
        \big| \nabla G_\lambda(x(\theta) - y_\Sigma - \varepsilon \theta s \nu(y_\Sigma)) \big| 
        \mathrm{d} \theta \mathrm{d} s \mathrm{d} \sigma(y_\Sigma) \leq C \big(1 + |\ln \varepsilon|\big).
  \end{equation*}
\end{lem}
\begin{proof}
	Let $x \in C([0, 1]; \mathbb{R}^d)$ be fixed and choose $r_0 > 0$. 	
	Interchanging the order of integration we get
	\begin{equation*}
     \begin{split}
    	\int_{\Sigma} \! \int_{-1}^1 \! \int_0^1 
	   \big| \nabla G_{\lambda}(x(\theta) - y_{\Sigma} - \eps \theta s \nu(y_{\Sigma})) 
        \big| \dd \theta \dd s \dd \sigma(y_\Sigma) \!=\! 
	\int_0^1 \! \int_{\Sigma} \! \int_{-1}^1 
	   \big| \nabla G_{\lambda}(x(\theta) - y_{\Sigma} - \eps \theta s \nu(y_{\Sigma})) 
        \big|  \dd s \dd \sigma(y_\Sigma) \dd \theta.
     \end{split}
	\end{equation*}
	We are going to find a suitable bound for the integral with respect to $\mathrm{d} s$ and $\mathrm{d} \sigma$.
	So let $\theta \in (0, 1)$ be fixed.
	If $\mathrm{dist}(x(\theta), \Sigma) > r_0$, 
	Lemma \ref{lemma2} yields immediately that
	\begin{equation*}
     \begin{split}
    	 \int_{\Sigma} \int_{-1}^1
	   \big| \nabla G_{\lambda}(x(\theta) - y_{\Sigma} - \eps \theta s \nu(y_{\Sigma})) 
        \big|  \dd s \dd \sigma(y_\Sigma) 
	\leq C_1 
     \end{split}
	\end{equation*}
	for all $\varepsilon > 0$ sufficiently small and thus, the claimed result is true in this case.

	If $\mathrm{dist}(x(\theta), \Sigma) \leq r_0$ we split the hypersurface $\Sigma$ into two disjoint
	parts 
	\[
		\Sigma_1 := \Sigma\cap B(x(\theta),r_0)\quad 
		\text{and} 
	  	\quad
	  	\Sigma_2 := \Sigma \setminus B(x(\theta),r_0), \quad \Sigma = \Sigma_1\cup\Sigma_2,
	\]
       and we define the following two auxiliary quantities
	\[
		\cI_{j}(x(\theta), \theta, \eps) := \int_{\Sigma_{j}} \int_{-1}^1 
	   \big| \nabla G_{\lambda}(x(\theta) - y_{\Sigma} - \eps \theta s \nu(y_{\Sigma})) 
        \big| \dd s \dd \sigma(y_\Sigma), \quad j \in \{ 1, 2 \}.
	\]
	Hence, it holds that
	\begin{equation}\label{eq:Ipm}
	   \int_{\Sigma} \int_{-1}^1 
	   \big| \nabla G_{\lambda}(x(\theta) - y_{\Sigma} - \eps \theta s \nu(y_{\Sigma})) 
        \big| \dd s \dd \sigma(y_\Sigma) = \cI_1(x(\theta), \theta, \eps) + \cI_2(x(\theta), \theta, \eps).
	\end{equation}
	Again Lemma \ref{lemma2} implies that $\cI_2(x(\theta), \theta, \eps) \leq C_1$ for all sufficiently small $\varepsilon > 0$
	independent of $x(\theta)$ and~$\theta$.
	
	It remains to estimate $\cI_1(x(\theta), \theta, \eps)$.
	Let $\{ \varphi_i, U_i, V_i \}_{i \in I}$ be a parametrization of $\Sigma$ as in Definition \ref{definition_hypersurface}.
  	By~\eqref{integral} we get
  	\begin{equation*}
      	\cI_1(x(\theta), \theta, \eps) 
	   	= 
	   	\sum_{i \in I} \int_{-1}^1 \int_{\varphi_i^{-1} (\Sigma_1)} \chi_i(\varphi_i(v))
        \big| 
        \nabla G_{\lambda}\big(x(\theta) - \varphi_i(v) 
        - \eps \theta s \nu(\varphi_i(v)) \big) \big| 
        \sqrt{\det G_i(v)} \dd v \dd s,
  	\end{equation*}
	where $\{ \chi_i \}_{i \in I}$ is a partition of unity subordinate to $\{ V_i \}_{i \in I}$.
    Since $\det G_i(v)$ is bounded by
    Definition~\ref{definition_hypersurface}~(e),
    we can continue estimating
	\begin{equation} \label{estimate_prop_diff1}
	\begin{split}
    	\cI_1(x(\theta), \theta, \eps) 
    	&=
    	\sum_{i \in I} \int_{-1}^1\int_{\varphi_i^{-1}(\Sigma_1)}
    	\chi_i(\varphi_i(v)) \big| \nabla G_{\lambda}\big(x(\theta) - \varphi_i(v) 
    	- \eps\tt s \nu(\varphi_i(v)) \big) \big|
        \sqrt{\det G_i(v)} \dd v \dd s \\
      	&\leq 
      	C_2 \sum_{i \in I} \int_{-1}^1\int_{\varphi_i^{-1} (\Sigma_1)}
        \big| \nabla G_{\lambda}\big(x(\theta) - \varphi_i(v) - 
        \eps\tt s \nu(\varphi_i(v)) \big) \big|
        \dd v \dd s \\
      &\leq C_3 \sum_{i \in I} \int_{-1}^1
       \int_{\varphi_i^{-1} (\Sigma_1)}
        \big| x(\theta) - \varphi_i(v) - \eps\tt s \nu(\varphi_i(v)) \big|^{1-d}
        \dd v \dd s,
    \end{split}
 	\end{equation}
	where we used an estimate for $|\nabla G_\lambda|$ that follows 
	from Proposition~\ref{proposition_asymptotics_G_lambda} (ii).
 	It remains to find a suitable bound for 
	$\big| x(\theta) - \varphi_i(v) - \eps\tt s \nu(\varphi_i(v)) \big|^{1-d}$.

	{\rm \bf Claim.} Let $i \in I$ with $\varphi_i^{-1}(\Sigma_1) \neq \emptyset$ be fixed.
	Then there exists a constant $\widetilde{C} = \widetilde{C}(i, \Sigma)$, $u = u(x(\theta), \theta, \varepsilon, i) \in \overline{U}_i$ 
	and $t = t(x(\theta), \theta, \varepsilon, i) \in [-1, 1]$ with the following properties:
	\begin{itemize}
		\item[(a)] for any $(v, s) \in \varphi_i^{-1}(\Sigma_1) \times (-1, 1)$ it holds
		\begin{equation*}
		|x(\theta) - \varphi_i(v) - \varepsilon \theta s \nu(\varphi_i(v))|^{1-d} 
		\leq \widetilde{C} (|u - v|^{d-1} + (\varepsilon \theta |t - s|)^{d-1})^{-1};
		\end{equation*}
		\item[(b)] $\varphi_i^{-1}(\Sigma_1) \subset B(u, 2 r_0/c)$.
	\end{itemize}
	Here, $c$ is the same constant as in Hypothesis \ref{hypothesis_hypersurface} (c).
	We mention that $u$ and $t$ also depend on $d$, $\Sigma$ and~$r_0$, but this is not of importance in the following.

	In order to prove the claim, set 
	\[
		\sfh := \inf \big\{ | x(\theta) - \varphi_i(v) - \eps \theta s \nu(\varphi_i(v)) |
		\colon
		 (v, s) \in \varphi_i^{-1}(\Sigma_1) \times (-1, 1) \big\}.
	\]	 
	Choose a sequence $(u_n, t_n) \in \varphi_i^{-1}(\Sigma_1) \times (-1, 1)$ 
	such that
  	\[
  		| x(\theta) - \varphi_i(u_n) - \eps \theta t_n \nu(\varphi_i(u_n)) | \arr \sfh
  		\quad\text{and}\quad
		\varphi_i(u_n) + \eps \theta t_n \nu(\varphi_i(u_n)) \arr y\in \dR^d;
  	\]
	this is possible since the set
  	$\{ \varphi_i(v) + \varepsilon \theta s \nu(\varphi_i(v))\colon 
  	(v, s) \in \varphi_i^{-1}(\Sigma_1) \times (-1, 1) \}$ is bounded in $\mathbb{R}^{d}$.
	Using Hypothesis~\ref{hypothesis_hypersurface}~(c),
	we find that
	\begin{equation*}
	    c^2 \big( |u_n - u_m|^2 + (\eps \theta)^2 |t_n - t_m|^2 \big)
	    \leq 
	    |\varphi_i(u_n) + \eps \theta t_n \nu(\varphi_i(u_n)) - \varphi_i(u_m) - \eps \theta t_m \nu(\varphi_i(u_m))|^2.
	\end{equation*}
	Therefore, $(u_n)$ and $(t_n)$ are Cauchy sequences.
	Set $\lim u_n =: u \in \ov{U_i}$ 
	and $\lim t_n =: t \in [-1, 1]$. 
	Using a continuity argument and again 
	Hypothesis \ref{hypothesis_hypersurface} (c), 
	we find for all $v \in U_i$ and all $s \in (-1, 1)$ that
	\begin{equation} \label{prop_app_3_label2}
    \begin{split}
    	c^2\big( |u - v|^2 &+ (\varepsilon \theta)^2 |t - s|^2 \big)
        = 
        c^2 \lim_{n \arr \infty} 
        \big( |u_n - v|^2 + (\eps \theta)^2 |t_n - s|^2 \big) \\
      	&\leq 
      	\lim_{n \arr \infty} 
      	|\varphi_i(u_n) + \eps \theta t_n \nu(\varphi_i(u_n)) - \varphi_i(v) - 
      	\eps \theta s \nu(\varphi_i(v))|^2
	    = 
	    |y - \varphi_i(v) - \eps \theta s \nu(\varphi_i(v))|^2.
    \end{split}
	\end{equation}
	Finally, 
	for $v \in \varphi_i^{-1}(\Sigma_1)$ and $s \in (-1, 1)$ 
	it holds by the triangle inequality
	\begin{equation*}
    	|y - \varphi_i(v) - \varepsilon \theta s \nu(\varphi_i(v))| 
        \leq |x(\theta) - \varphi_i(v) - \varepsilon \theta s \nu(\varphi_i(v))| + |x(\theta) - y|
        \leq 2 |x(\theta) - \varphi_i(v) - \varepsilon \theta s \nu(\varphi_i(v))|
	\end{equation*}
	due to the construction of $y$. Hence, we get that
	\begin{equation} \label{prop_app_3_label3}
		|x(\theta) - \varphi_i(v)- \varepsilon \theta s \nu(\varphi_i(v))| 
		\geq \frac{1}{2} |y - \varphi_i(v) - \varepsilon \theta s \nu(\varphi_i(v))|.
	\end{equation}
	In particular, this and \eqref{prop_app_3_label2} with $s=0$ imply that
	\begin{equation*}
		c |u - v| \leq c \big( |u - v|^2 + (\varepsilon \theta t)^2 \big)^{1/2} \leq |y - \varphi_i(v)| \leq 2 |x(\theta) - \varphi_i(v)| \leq 2 r_0
	\end{equation*}
	and thus, $\varphi_i^{-1} (\Sigma_1) \subset B(u, 2 r_0/c)$.
	This is item (b) of the claim.
	Furthermore, using for $a, b > 0$ the inequality
	\begin{equation*}\label{eq:simple_inequality}
		a^{d-1} + b^{d-1} \le (a+b)^{d-1} \quad \Leftrightarrow \quad
		(a + b)^{1-d} \leq \big( a^{d-1} + b^{d-1} \big)^{-1}
	\end{equation*}	
	and equations \eqref{prop_app_3_label2} and \eqref{prop_app_3_label3} we obtain
	\begin{equation*}
    \begin{split}
    	\big| x(\theta) - \varphi_i(v)  - \eps \theta s \nu(\varphi_i(v)) \big|^{1-d}
        & \leq 
        2^{d-1} 
        \big| y - \varphi_i(v)  - \eps \theta s \nu(\varphi_i(v)) \big|^{1-d} 
	    \leq 
	    C_4\big( |u - v|^2 + (\eps \theta)^2 |t - s|^2 \big)^{(1-d)/2} \\
	    &\leq 
	    C_5
	    \big( |u - v| + \eps \theta |t - s| \big)^{1-d} 
	    \leq 
	    C_5 \left(|u - v|^{d-1} + 
	    (\eps \theta |t - s|)^{d-1}\right)^{-1}.
    \end{split}
	\end{equation*}
	Thus, also assertion (a) of the claim is true.

	Using the result of the above claim, we can continue with the estimate \eqref{estimate_prop_diff1}.
	Employing a transition to spherical coordinates and that
	$\varphi_i^{-1} (\S_1) \subset B(u, 2 r_0/c)$, we get 
      for all sufficiently small $\eps >0$ that
	\begin{equation*}
	\begin{split}
    	\int_{-1}^1
      	\int_{\varphi_i^{-1} (\S_1)} 
        \big| x(\theta) - \varphi_i(v) - \eps\tt s \nu(\varphi_i(v)) \big|^{1-d}
        \dd v \dd s \!
        & \leq \! C_5
       	\int_{-1}^1
      	\int_{B(u,2 r_0/c)} 
	    \left(|u - v|^{d-1} + 
	    (\eps \theta |t - s|)^{d-1}\right)^{-1}
        \dd v \dd s\\
        & = C_{6}
        \int_{-1}^1
      	\int_0^{2 r_0/c} 
	    r^{d-2}\left(r^{d-1} + (\eps \theta |t - s|)^{d-1}\right)^{-1}
        \dd r \dd s.
    \end{split}    
	\end{equation*}
	We employ now the substitutions $z := r^{d-1}$ and $\zeta := s - t$.
	Note that $s, t \in [-1, 1]$ implies $\zeta \in [-2, 2]$. Hence, we get
	\begin{equation*}
	\begin{split}
    	\int_{-1}^1
      	\int_{\varphi_i^{-1} (\S_1)} 
        \big| x(\theta) - \varphi_i(v) - \eps\tt s \nu(\varphi_i(v)) &\big|^{1-d} 
        \dd v \dd s
       \leq C_{6}
  	 \int_{-1}^1
      	\int_0^{2 r_0/c} 
	    r^{d-2}\left(r^{d-1} + (\eps \theta |t - s|)^{d-1}\right)^{-1}
        \dd r \dd s \\
        &\leq C_{7}
        \int_{-2}^2 \int_0^{(2 r_0/c)^{d-1}} 
	    \left(z + (\eps \theta |\zeta|)^{d-1}\right)^{-1}
        \dd z\dd \zeta \\
        &= C_{7}
        \int_{-2}^2 \bigg( \ln\bigg( \bigg(\frac{2 r_0}{c}\bigg)^{d-1} + (\eps \theta |\zeta|)^{d-1}\bigg) 
        - \ln \big((\eps \theta |\zeta|)^{d-1} \big)\bigg) \dd \zeta.
	\end{split}
	\end{equation*}
	For a fixed $r_0 > 0$ it holds that $\ln\left( \left(\frac{2 r_0}{c}\right)^{d-1} + (\eps \theta |\zeta|)^{d-1}\right) \leq C_{8}$
	independent of $\theta \in (0, 1)$. Therefore, we obtain
	\begin{equation*}
	\begin{split}
    	\int_{-1}^1
      	\int_{\varphi_i^{-1} (\S_1)} 
        \big| x(\theta) - \varphi_i(v) - \eps\tt s \nu(\varphi_i(v)) \big|^{1-d}
        \dd v \dd s
         &\leq C_{9}
        \int_{-2}^2 \big( 1 - \ln (\eps \theta |\zeta|)\big)
        \dd \zeta
           \le C_{10}\big( 1 + |\ln \eps| + |\ln \theta| \big).
    \end{split}    
	\end{equation*}
	Summing up all $i \in I$, we get
	\begin{equation*}
	\begin{split}
		\cI_1(x(\theta), \theta, \eps) 
  	  	&\leq C_3 \sum_{i \in I} \int_{-1}^1
      		 \int_{\varphi_i^{-1} (\Sigma_1)}
      		  \big| x(\theta) - \varphi_i(v) - \eps\tt s \nu(\varphi_i(v)) \big|^{1-d}
        	\dd v \dd s 
		\leq C_{11}\big( 1 + |\ln \eps| + |\ln \theta| \big).
	\end{split}
	\end{equation*}

	Thus, using \eqref{eq:Ipm} and the last estimate, it follows
	\begin{equation*}
     \begin{split}
    	\int_{\Sigma} \int_{-1}^1 
        \big| \nabla G_\lambda(x(\theta) - y_\Sigma - \varepsilon \theta s \nu(y_\Sigma)) \big| 
         \mathrm{d} s \mathrm{d} \sigma(y_\Sigma) 
	   &= \big(\cI_1(x(\theta), \theta, \eps) + \cI_2(x(\theta), \theta, \eps) \big)  
	   \leq C_{12} \big( 1 + |\ln \eps| + |\ln \theta| \big).
     \end{split}
	\end{equation*}
	This is the desired estimate, if $\mathrm{dist}(x, \Sigma) \leq r_0$.
	Integrating the last estimate with respect to $\theta$ we finally obtain
	\begin{equation*}
     \begin{split}
    	\int_0^1\int_{\Sigma} \int_{-1}^1 
        \big| \nabla G_\lambda(x(\theta) - y_\Sigma - \varepsilon \theta s \nu(y_\Sigma)) \big| 
         \mathrm{d} s \mathrm{d} \sigma(y_\Sigma) \mathrm{d} \theta
	    &\leq C_{13} \int_0^1\big( 1 + |\ln \eps| + |\ln \theta| \big)\mathrm{d} \theta
		\leq C_{14} \big( 1 + |\ln \varepsilon|\big),
     \end{split}
	\end{equation*}
	which leads to the statement of this lemma.
\end{proof}

In the following proposition we state the last two estimates that are needed to prove our main result.
They follow from Lemma \ref{proposition_difficult}, that is once applied for the case that
$x(\theta)$ is a constant function and once for $x(\theta)$ being linear affine in $\theta$.

\begin{prop} \label{proposition_appendix_3}
	Let $\Sigma$ be a $C^2$-smooth hypersurface satisfying 
	Hypothesis~\ref{hypothesis_hypersurface} and
	let $\lambda \in (-\infty, 0)$.
	\begin{itemize}
	\item[\rm (i)] Let
	$\psi \in L^{\infty}(\Sigma \times (-1, 1))$. 
	Then there exists a constant 
	$C = C(d,\lm,\S,\psi) > 0$
	such that the bound
	\begin{equation*}
	   \sup_{x \in \dR^d} 
	   \int_{\Sigma} \int_{-1}^1 
	   \big| \big(
          G_{\lambda}(x - y_{\Sigma} - \eps s \nu(y_{\Sigma})) 
          - 
          G_{\lambda}(x - y_{\Sigma}) 
       \big) \psi(y_\Sigma, s) \big| \dd s \dd \sigma(y_\Sigma) 
       \leq C \eps\big( 1 + |\ln \eps| \big)
	\end{equation*}
	holds for all sufficiently small $\eps > 0$.
	\item[\rm (ii)] Let
	$\omega, \psi \in L^{\infty}(\Sigma \times (-1, 1))$. 
	Then there exists a constant
	$C = C(d,\lm,\S,\omega,\psi) > 0$ 
	such that the bound
	\begin{equation*}
     \begin{split}
	\sup_{(x_\Sigma, t) \in \Sigma \times (-1, 1)} 
    	\int_{\Sigma} \int_{-1}^1 
    	\big| \omega(x_\Sigma, t) 
    	\big(
  	      G_{\lambda}(x_\Sigma + \eps t \nu(x_\Sigma) 
  	     	 - y_{\Sigma} -  \eps s \nu(y_{\Sigma}))& \\
  		- G_{\lambda}(x_\Sigma - y_{\Sigma}) 
          \big) \psi(y_\Sigma, s) &\big| \dd s \dd \sigma(y_\Sigma) 
          \leq C \eps \big(1+|\ln \eps|\big)
     \end{split}
	\end{equation*}
	holds for all sufficiently small $\eps > 0$.
	\end{itemize}
\end{prop}
\begin{proof}
	(i) Let $x \in \mathbb{R}^d$ be fixed. 	
	We employ \eqref{estimate_diff_G_lambda} 
	(with $\tau(x, y_\Sigma, s, t) = -s \nu(y_{\Sigma})$ implying $\| \tau \|_{L^\infty} = 1$)
	and find that
	\begin{equation*} \label{eq:start1}
     \begin{split}
    	\int_{\Sigma} \int_{-1}^1 
	   \big| \big(
          G_{\lambda}(x - y_{\Sigma} - \eps s \nu(y_{\Sigma})) 
          &- G_{\lambda}(x - y_{\Sigma}) 
       \big) \psi(y_\Sigma, s) \big| \dd s \dd \sigma(y_\Sigma)  \\
	&\leq \varepsilon \| \psi \|_{L^\infty} \int_{\Sigma} \int_{-1}^1 \int_0^1
	   \big| \nabla G_{\lambda}(x - y_{\Sigma} - \eps \theta s \nu(y_{\Sigma})) 
        \big| \dd \theta \dd s \dd \sigma(y_\Sigma).
     \end{split}
	\end{equation*}
	Thus, Lemma \ref{proposition_difficult} yields the claimed assertion.

	(ii) Let $(x_\Sigma, t) \in \Sigma \times (-1, 1)$
	be fixed.
	Using \eqref{estimate_diff_G_lambda} 
	(with $x = x_\Sigma$ and $\tau(x_\Sigma, y_\Sigma, s, t) = t \nu(x_\Sigma) - s \nu(y_{\Sigma})$
    implying $\| \tau \|_{L^\infty} \leq 2$)
	we find that
	\begin{equation*} \label{eq:start2}
     \begin{split}
    	\int_{\Sigma} \int_{-1}^1 &
    	\big| \omega(x_\Sigma, t) 
    	\big(
  	      G_{\lambda}(x_\Sigma + \eps t \nu(x_\Sigma) 
  	     	 - y_{\Sigma} -  \eps s \nu(y_{\Sigma}))
  		- G_{\lambda}(x_\Sigma - y_{\Sigma}) 
          \big) \psi(y_\Sigma, s) \big| \dd s \dd \sigma(y_\Sigma) \\
	&\leq 2 \varepsilon \| \omega \|_{L^\infty} \| \psi \|_{L^\infty} \int_{\Sigma} \int_{-1}^1 \int_0^1
    	\big| \nabla G_\lambda\big(x_\Sigma - y_\Sigma + \varepsilon \theta (t \nu(x_\Sigma) - s \nu(y_\Sigma))\big)
          \big| \dd \theta \dd s \dd \sigma(y_\Sigma).
     \end{split}
	\end{equation*}
	Therefore, the claimed statement follows from Lemma \ref{proposition_difficult} 
	(with $x(\theta) = x_\Sigma + \varepsilon \theta t \nu(x_\Sigma)$).
\end{proof}

\section{Boundaries of bounded $C^2$-domains} \label{appendix_boundary}

In this appendix it is shown that the boundary of a bounded and simply connected $C^2$-domain is a $C^2$-hypersurface 
in the sense of Definition~\ref{definition_hypersurface} that satisfies Hypothesis~\ref{hypothesis_hypersurface}. 
Recall that a bounded $C^2$-domain $\Omega \subset \mathbb{R}^d$ is an open, bounded and simply connected set 
with boundary $\Sigma$ for which there exists a parametrization $\{\Phi_i, X_i, Y_i\}_{i\in I}$ with a finite index set $I$ satisfying 
the following conditions:
\begin{itemize}\skipex
	\item [\rm (a)] 
	$X_i \subset \dR^{d-1}$ 
	and $Y_i \subset \dR^d$ 
	are open sets and  
	$\Phi_i\colon X_i \rightarrow Y_i$ 
	is a $C^2$-mapping for all $i \in I$;
	\item[\rm (b)]
	 $\rank D \Phi_i(u) = d - 1$ 
	 for all $u \in X_i$ and $i\in I$;
    \item[\rm (c)] 
    $\Phi_i(X_i) = Y_i \cap \Sigma$ 
    and $\Phi_i\colon X_i \rightarrow Y_i \cap \Sigma$ 
    is a homeomorphism;
    \item[\rm (d)] $\Sigma \subset \bigcup_{i \in I} Y_i$;
\end{itemize}
see, e.g. in \cite[Chapter 3]{mc_lean} for a similar notion.
Since $\Sigma$ is compact it is no restriction to assume that the sets $Y_i$, $i\in I$, are bounded. 
In order to prove that $\Sigma$ is a hypersurface in the sense of Definition~\ref{definition_hypersurface}
choose a family $\{V_i\}_{i\in I}$
of open sets in $\dR^d$ such that 
$V_i \subset \ov{V_i}\subset Y_i$ and $\Sigma \subset\bigcup_{i\in I} V_i$.
Next, set $U_i := \Phi_i^{-1}(V_i\cap\Sigma)\subset X_i$ 
and $\varphi_i := \Phi_i|_{U_i}$. 
Then $U_i \subset \overline{U_i} \subset X_i$ and $\{\varphi_i, U_i, V_i\}_{i\in I}$  satisfies the conditions (a)-(d) above.
In addition, each $\varphi_i$ is uniformly Lipschitz continuous on $U_i$,
as $\overline{U_i} \subset X_i$.
Therefore, $\{\varphi_i, U_i, V_i\}_{i\in I}$ is a parametrization of $\Sigma$
with the properties (a)-(e) in Definition \ref{definition_hypersurface} and hence $\Sigma$ is a $C^2$-hypersurface. 

Moreover, as $\overline{U_i} \subset X_i$ there is a constant $C > 0$ such that the corresponding 
first fundamental form in~\eqref{Gi} satisfies
\begin{equation} \label{first_fundamental_form_domain}
  \det G_i(u) \geq C > 0
\end{equation}
for any $u \in U_i$ and all $i \in I$.
Since the outward unit normal vector field $\nu$ is continuous, it is clear that the hypersurface $\Sigma$ 
is orientable.

Recall that the eigenvalues of the 
Weingarten map depend continuously on $x_{\Sigma}\in\Sigma$ and, since $\Sigma$ is compact, these eigenvalues are bounded. 
Hence, item (b) of Hypothesis~\ref{hypothesis_hypersurface} is satisfied. 
It remains to show that (a) and (c)
of Hypothesis~\ref{hypothesis_hypersurface} hold. 
The following proposition concerns condition (a).

\begin{prop} \label{proposition_iota_injective}
Let $\Omega\subset\dR^d$ be a bounded and simply connected $C^2$-domain with boundary $\Sigma$. Then there exists $\beta>0$ such that the mapping 
$\iota_\Sigma:\Sigma\times\dR\rightarrow\dR^d$ in \eqref{iota} is injective on $\Sigma \times (-\beta, \beta)$. 
\end{prop}
\begin{proof}
Suppose that the claim is false. 
Then for all $n\in\dN$ there exist 
$x_{\Sigma, n},y_{\Sigma, n} \in\Sigma$ 
and $s_n,t_n\in \big(-\frac{1}{n},\frac{1}{n}\big)$ such that
$(x_{\Sigma, n},t_n)\ne (y_{\Sigma, n},s_n)$ and 
\begin{equation}
\label{eq:xnynanbn}
x_{\Sigma, n} + t_n\nu(x_{\Sigma, n}) = \iota_\Sigma(x_{\Sigma, n}, t_n) 
= \iota_\Sigma(y_{\Sigma, n}, s_n) = y_{\Sigma, n} + s_n\nu(y_{\Sigma, n}).
\end{equation}
Since $\Sigma$ is compact we can assume that the sequences $(x_{\Sigma, n})_n$ and $(y_{\Sigma, n})_n$ converge to $x_\Sigma$ and $y_\Sigma$, respectively.
Then, equation \eqref{eq:xnynanbn} implies $x_\Sigma = y_\Sigma$.
Let $\{\varphi_i, U_i, V_i\}_{i\in I}$ be the parametrization of $\Sigma$ constructed above
and let $\iota_{\varphi_i}:U_i\times \dR\rightarrow\dR^d$
be as in \eqref{iotavarphi}. 
Since $x_{\Sigma, n} \rightarrow x_\Sigma$ and $y_{\Sigma, n}\rightarrow y_\Sigma = x_\Sigma$, 
there exists $N\in\dN$ and $i\in I$ such that $x_{\Sigma, n},y_{\Sigma, n} \in V_i$ 
for all $n\ge N$.
Hence, $u_n := \varphi_i^{-1}(x_{\Sigma, n})$ and 
$v_n := \varphi_i^{-1}(y_{\Sigma, n})$, $n \ge N$, satisfy
\[
\iota_{\varphi_i}(u_n,t_n) 
= 
\varphi_i(u_n) + t_n\nu(\varphi_i(u_n))
= 
\varphi_i(v_n) + s_n\nu(\varphi_i(v_n)) 
=
\iota_{\varphi_i}(v_n,s_n)
\]
and it follows that $\iota_{\varphi_i}$ is not injective on $U_i\times (-\beta,\beta)$ 
for any $\beta > 0$. 
On the other hand, if $G_i$ and $L_i$ denote the matrices representing the first fundamental form and the Weingarten map, then
\begin{equation*}
|\det D\iota_{\varphi_i}(u,t)| = \det(1 - tL_i(u))\sqrt{\det G_i(u)} \ge C > 0
\end{equation*}
holds for some $C >0$, all $u\in U_i$ and all $t$ sufficiently small; cf. \eqref{DetDiota} and \eqref{first_fundamental_form_domain}.
Now the inverse function theorem implies that $\iota_{\varphi_i}$
is invertible; a contradiction. 
\end{proof}

Finally, we prove that condition {\rm (c)} of 
Hypothesis~\ref{hypothesis_hypersurface} is satisfied for the boundary of a compact and simply connected $C^2$-smooth domain.

\begin{prop}
Let $\Omega\subset\dR^d$ be a bounded and simply connected $C^2$-domain with boundary $\Sigma$,
let $\{\varphi_i, U_i, V_i\}_{i\in I}$ be the parametrization of $\Sigma$ as above and let
$\iota_{\varphi_i}$ be as in \eqref{iotavarphi}. Then
    there exist $\beta>0$ and a constant $c > 0$  
    such that for all $u, v \in U_i$, $s, t \in (-\beta, \beta)$ and all $i\in I$ it holds
    \[
     \big|\iota_{\varphi_i}(u,t) - \iota_{\varphi_i}(v,s)\big|^2 
      \ge c^2\left( |u - v|^2 + |s -  t|^2 \right).
    \]
\end{prop} 

\begin{proof}
Let $\{ \Phi_i, X_i, Y_i \}_{i \in I}$ be a parametrization of $\Sigma$ as 
in the beginning of this appendix, fix $i \in I$ and 
choose $\beta > 0$ and a constant $D > 0$ such that
\begin{equation}
\label{eq:uniformpositive}
\det(1 - tL_i(u)) \ge D > 0\quad\text{for all}\quad (u,t)\in X_i \times (-2\beta,2\beta);
\end{equation} 
here $L_i(u)$ is the matrix representing the Weingarten 
map associated to the mapping $\Phi_i$.  
As in \eqref{iotavarphi} define
\begin{equation*}
\iota_{\Phi_i}\colon X_i \times \dR  \rightarrow \dR^d, 
    \qquad \iota_{\Phi_i}(u,t) :=  \Phi_i(u) + t \nu(\Phi_i(u)),
\end{equation*}
and consider the open sets 
\[
\Omega_1 := \iota_{\varphi_i}(U_i\times (-\beta,\beta))
\quad\text{and}\quad 
\Omega_3 := \iota_{\Phi_i}(X_i\times (-2\beta,2\beta))
\]
in $\dR^d$.
Since $\Omega_1\subset\ov{\Omega_1}\subset\Omega_3$,
there exists $\beta_1 > 0$ such that 
$\ov{B(x, 2 \beta_1)}\subset\Omega_3$ 
for any $x \in \Omega_1$. Eventually, we need the open set
\begin{equation*}
  \Omega_2 := \bigcup_{x \in \Omega_1} B(x, \beta_1).
\end{equation*}
Note that $\Omega_1 \subset \Omega_2 \subset \overline{\Omega_2} \subset \Omega_3$.

Let $u, v\in U_i$ and $s,t\in(-\beta,\beta)$ be fixed. 
Set $x := \iota_{\varphi_i}(u,t)=\iota_{\Phi_i}(u,t)$ and $y  := \iota_{\varphi_i}(v,s)=\iota_{\Phi_i}(v,s)$.
We distinguish two cases: 
$|x - y| < \beta_1$ and $|x - y| \geq \beta_1$. 
In the first case, if $|x - y| < \beta_1$ then $y$ 
is contained in the convex set $B(x, \beta_1) \subset \Omega_2$. 
It follows from
\begin{equation*}
\big|\det \big((D \iota_{\Phi_i})(u, t)\big)\big|
=  \big| \det(1 - t L_i(u)) \big| \sqrt{\det G_i(s)}
\end{equation*}
(cf. \eqref{DetDiota}), \eqref{eq:uniformpositive} and \eqref{first_fundamental_form_domain}, which is also true
for $u\in \Omega_2$ with a possibly smaller constant $C$,
that $|\det D\iota_{\Phi_i}(u,t)| \ge D' >0$
for all $(u,t)\in \Phi_i^{-1}(\Omega_2)$.
Hence, there exists a constant $C_1 >0$ such that
\[
\big\|(D\iota_{\Phi_i}^{-1})(z)\big\| = \left\| \big(D \iota_{\Phi_i} 
(\iota_{\Phi_i}^{-1}(z)) \big)^{-1} \right\|=
\left\| \frac{1}{\det D \iota_{\Phi_i} 
(\iota_{\Phi_i}^{-1}(z)) }\,\text{adj}\bigl[ D \iota_{\Phi_i} 
(\iota_{\Phi_i}^{-1}(z))\bigr] \right\|
\le C_1 
\]
for all $z \in \Omega_2$; here $\text{adj}\,[\cdot]$ denotes the adjugate matrix.
Hence, we obtain
\begin{equation*}
  \begin{split}
    \bigl( |u - v|^2 + |s - t|^2 \bigr)^{1/2} &= \big|\iota_{\Phi_i}^{-1}(y) - \iota_{\Phi_i}^{-1}(x)\big| = 
    \left| \int_0^1 \frac{\mathrm{d}}{\mathrm{d} \xi} \iota_{\Phi_i}^{-1}(x + \xi (y - x)) \mathrm{d} \xi \right| \\
    &\leq \int_0^1 \big|D \iota_{\Phi_i}^{-1}(x + \xi (y - x))\! \cdot \!(y - x)\big| \mathrm{d} \xi 
    \leq \max\left\{ \big\| D \iota_{\Phi_i}^{-1}(z) \big\|\!:\! z \in B(x, \beta_1) \right\} \!\cdot\! | x - y | \\
    &\leq C_1 | \iota_{\varphi_i}(u, t) - \iota_{\varphi_i}(v, s) |,
  \end{split}
\end{equation*}
which is the claimed result in the case $|x - y| < \beta_1$.

In the remaining case $|x - y| \geq \beta_1$ we set
$C_2 := \max\{|\iota_{\varphi_i}^{-1}(z)|\colon z\in\Omega_1\}$.
Since $\overline{\Omega}_1 \subset \Omega_3$ we conclude $C_2 < \infty$ and
\[
|u- v|^2 + |s-t|^2
= 
\big|\iota_{\varphi_i}^{-1}(x) - \iota_{\varphi_i}^{-1}(y)\big|^2 
\le
\frac{4C_2^2}{|x-y|^2}|x-y|^2
\le 
\frac{4C_2^2}{\beta_1^2}\big|\iota_{\varphi_i}(u,t) - \iota_{\varphi_i}(v,s)\big|^2. 
\]
Setting finally
\[
C^{(i)}:= \min\bigg\{\frac{1}{C_1},\frac{\beta_1}{2C_2}\bigg\}
\]
and $c := \min \big\{ C^{(i)}: i \in I \big\}$, the result of
this proposition follows.
\end{proof}

\end{appendix}



\end{document}